\documentclass[11pt,a4paper]{article}
\usepackage{amsmath, amssymb, enumitem, dsfont}
\usepackage[a4paper, left=2.5cm, right=2.5cm, top=2.5cm, bottom=2.5cm]{geometry}

\usepackage{caption} 
\usepackage{xcolor} 
\usepackage{amsthm,thmtools} 
\usepackage{graphicx}

\usepackage[colorlinks=true, linkcolor=black, citecolor=black, urlcolor=black]{hyperref}
\usepackage{cleveref}
\crefformat{theorem}{{Theorem~#2#1#3}}
\crefformat{corollary}{{Corollary~#2#1#3}}
\crefformat{proposition}{{Proposition~#2#1#3}}
\crefformat{lemma}{{Lemma~#2#1#3}}
\crefformat{definition}{{Definition~#2#1#3}}
\crefformat{remark}{{Remark~#2#1#3}}
\crefformat{assumption}{{Assumption~#2#1#3}}
\crefformat{section}{{Section~#2#1#3}}
\crefformat{subsection}{{Subsection~#2#1#3}}

\definecolor{bluepoli}{cmyk}{0.4,0.1,0,0.4}

\declaretheoremstyle[
  headfont=\color{bluepoli}\normalfont\bfseries,
  bodyfont=\color{black}\normalfont\itshape,
]{colored}

\newtheorem{theorem}{Theorem}[section]

\newtheorem{corollary}[theorem]{Corollary}
\newtheorem{proposition}[theorem]{Proposition}
\newtheorem{lemma}[theorem]{Lemma}
\newtheorem{definition}[theorem]{Definition}
\newtheorem{remark}[theorem]{Remark}
\newtheorem{assumption}[theorem]{Assumption}

\numberwithin{equation}{section}

\title{\textbf{On a semilinear heat equation on infinite graphs I: blow-up for large initial data}}
\author{
Fabio Punzo\thanks{
Dipartimento di Matematica, Politecnico di Milano,
Piazza Leonardo da Vinci 32, 20133 Milano, Italia.
Email: \texttt{fabio.punzo@polimi.it}.
Membro INdAM e parzialmente supportato dal progetto GNAMPA 2026.
}
\and
Federico Zucchero\thanks{
Dipartimento di Matematica, Politecnico di Milano,
Piazza Leonardo da Vinci 32, 20133 Milano, Italia.
Email: \texttt{federico.zucchero@mail.polimi.it}.
}
}

\date{}

\begin{document}

\maketitle


\abstract{We investigate finite-time blow-up of solutions to the
Cauchy problem for a semilinear heat equation posed on
infinite graphs. Assuming that the initial datum is
sufficiently large, we establish a general blow-up
criterion valid on arbitrary infinite graphs.

We then apply this result to specific classes of graphs,
including trees and the integer lattice. The approach
developed in the paper can be regarded as a discrete
counterpart of Kaplan's method, suitably adapted to the
graph setting. In \cite{PZ2}, which is the second part of this
work, we also complement the blow-up analysis by addressing
arbitrary initial data and proving global existence for
sufficiently small data.
}




\section{Introduction}\label{section_introduction}
We investigate finite-time blow-up of solutions to the Cauchy problem
\begin{equation}\label{problem_HR_G}
\begin{cases}
u_t - \Delta u = f(u) & \text{in } X \times (0,T), \\
u = u_0 & \text{in } X \times \{0\},
\end{cases}
\tag{1.1}
\end{equation}
where $(X,\omega,\mu)$ is an infinite weighted graph with edge
weight $\omega$ and node measure $\mu$, $T>0$,
$\Delta$ denotes the weighted Laplacian on the graph. Furthermore, $f:[0, \infty)\to [0, \infty)$ is a locally Lipschitz and convex function fulfilling
$f(0)=0$, $f(u)>0$ for all $u>0$,
\[
\int_{1}^{\infty}\frac{du}{f(u)}<\infty
\qquad\text{and}\qquad
\lim_{u\to\infty}\frac{f(u)}{u}=+\infty,
\]
while  $u_0 : X \to \mathbb{R}$ is a given nonnegative
initial datum.

Problem  \eqref{problem_HR_G} is set in the framework of infinite weighted
graphs, which provide a natural discrete counterpart of
Euclidean spaces and Riemannian manifolds. The weighted
Laplacian reflects both the combinatorial structure of
the graph and the underlying measure, allowing one to
describe diffusion processes on discrete spaces in close
analogy with the continuous theory.

More recently, increasing attention has been devoted to
elliptic and parabolic equations posed on graphs.
Foundational references in this area include the
monographs \cite{Grig1, KLW, Mu}, along with numerous
contributions concerning elliptic and parabolic equations
(see, e.g., \cite{BMP, BP, FR, GMP2, GSXX, HK, Huang, KLW, LSZ, LW1, LW2, MePu, MPS1, MPS2, PSa, W, Woj, Wu}).

These works emphasize both the analogies with the
continuous setting and the distinctive features arising
from the discrete structure, such as volume growth
conditions, heat kernel estimates, and functional
inequalities.

The study of nonlinear diffusion equations on graphs is
strongly motivated by applications. Graphs naturally
model diffusion and reaction processes on networks,
including transportation systems, electrical circuits,
communication networks, and social or biological
interaction networks. In such contexts, the function $u$
may represent the density of a substance, the intensity
of a signal, or the concentration of a population
evolving on a discrete structure. Finite-time blow-up
corresponds to concentration or instability phenomena
and may describe threshold effects or breakdown
mechanisms in the underlying network dynamics.
Understanding whether solutions exist globally or blow up
in finite time is therefore of both theoretical and
applied interest.

In the Euclidean framework, the question of whether
problem  \eqref{problem_HR_G}  admits global solutions, or instead
exhibits nonexistence phenomena, has been thoroughly
investigated, particularly in the case $f(u) = u^p$
(see, for instance,  \cite{BB, BLev, DL, Fujita, H,  KST, Lev, MitPo, Weis}  and the
references therein). A cornerstone result is the
identification of the Fujita critical exponent, which
separates global existence from finite-time blow-up
for small initial data in $\mathbb{R}^N$.

When the ambient space is replaced by a Riemannian
manifold, the qualitative behavior of solutions may
change substantially, as the underlying geometry plays
a decisive role. This shift in perspective has been
explored in several works (see, e.g.,  \cite{BPT, GMP1, MMP, Pu1, Pu2, Z}).

For problem  \eqref{problem_HR_G} posed on an infinite graph, existence
of global solutions and finite-time blow-up of solutions
for any nontrivial initial datum have been investigated
in \cite{vCr, GMP2, GSXX, LSZ, LW1, LW2, MPS2, PSa, W}. These contributions
show that the critical behavior depends on structural
properties of the graph, such as volume growth and heat
kernel estimates, highlighting both similarities with
and differences from the Euclidean and manifold cases. Some basic results concerning blow-up of solutions for large initial data have been established in \cite{LW1}.

\smallskip

The aim of this paper is to establish finite-time blow-up
of solutions to problem \eqref{problem_HR_G} under the
assumption that the initial datum $u_0$ is sufficiently
large. More precisely, we provide a general criterion
ensuring finite-time blow-up of solutions to
\eqref{problem_HR_G}, which may be viewed as a version
of Kaplan's method adapted to infinite graphs.

In this respect, we extend to the graph setting the
results obtained in \cite{BLev, Pu2}, where the cases of
$\mathbb R^N$ and Riemannian manifolds were considered,
respectively. The discrete nature of the ambient space
requires substantial modifications of the arguments,
since several analytical tools available in the
continuous framework must be suitably reformulated
for infinite graphs.

After establishing the general blow-up criterion, we
apply it to several specific classes of graphs. We
first derive a result valid on a general infinite graph.
We then refine this result in the case of a generic
tree, and further specialize it to homogeneous trees.
Finally, we consider the integer lattice $\mathbb Z^N.$

In \cite{PZ2}, which constitutes the second part of this
work, we further develop these results by considering
arbitrary initial data on $\mathbb Z^N$. Moreover, we
establish global-in-time existence for sufficiently
small initial data, thereby complementing the blow-up
analysis and providing a more comprehensive description
of the qualitative behavior of solutions.

\smallskip

The paper is organized as follows. In Section
\ref{section_mathematical_framework} we introduce the
notation concerning infinite graphs and recall some
preliminary results. In Section \ref{section_hypotheses}
we present the main assumptions adopted throughout the
paper. The main results are stated in Section
\ref{section_main_results}. Sections
\ref{subsection_proofs_1} and \ref{subsection_proofs_2}
are devoted to the proofs in the general graph setting.
Finally, in Sections \ref{subsection_proofs_3} and
\ref{subsection_proofs_4} we establish the results for
trees and for the integer lattice, respectively.


\section{Mathematical framework}\label{section_mathematical_framework}


\subsection{The graph setting}\label{subsection_graph_setting}

\medskip
\begin{definition}\label{definition_weighted_graph}
Let \( X \) be an infinitely countable set and consider a function \( \mu : X \to (0,+\infty) \). Moreover, let \( \omega : X \times X \to [0,+\infty) \) be a map satisfying the following properties:
\begin{enumerate}
\renewcommand{\labelenumi}{\alph{enumi})}
    \item \( \omega(x,x) = 0 \) for all \( x \in X \);
    \item \( \omega \) is symmetric, that is, \( \omega(x,y) = \omega(y,x) \) for all \( x,y \in X \);
    \item \(\displaystyle \sum_{y \in X} \omega(x,y) < +\infty\) for all \( x \in X \).
\end{enumerate}
Then the triplet \( (X, \omega, \mu) \) is called weighted graph, and the functions \( \mu \) and \( \omega \) are referred to as vertex (or node) measure and edge weight, respectively. Two vertices \( x,y \in X \) are said to be adjacent (or equivalently connected, joint or neighbors) whenever \( \omega(x, y) > 0 \); in this case we write \( x \sim y \), and the pair \( (x,y) \) defines an edge of the graph with endpoints \( x,y \).
\end{definition}

\medskip
\noindent Observe that, since the weight function \( \omega \) determines the edges of a weighted graph, condition a) in \cref{definition_weighted_graph} implies the absence of \emph{self-loops}, that is, there are no edges of type \( (x,x) \). Moreover, condition b) in \cref{definition_weighted_graph} guarantees that the graph is \emph{undirected}, meaning that its edges do not have an orientation.

\medskip
\noindent The definition of the node measure \( \mu \) can be extended to the power set \( \mathcal{P}(X) \) as follows:
\[
\begin{split}
\mu(A) &:= \sum_{x \in A} \mu(x) \hspace{1.6em} \text{for every } A \in \mathcal{P}(X) \setminus \left\{ \emptyset \right\}, \\
\mu(\emptyset) &:= 0.
\end{split}
\]
\noindent Such extension defines a measure; in particular, \( (X, \mathcal{P}(X), \mu) \) is a measure space.

\medskip
\begin{definition}\label{definition_locally_finite}
The weighted graph \( (X, \omega, \mu) \) is said to be
\begin{enumerate}
\renewcommand{\labelenumi}{\alph{enumi})}
    \item locally finite if each vertex has only finitely many neighbors, namely
    \[
    \left| \{ y \in X : y \sim x \} \right| < +\infty, \qquad \quad \text{for all } x \in X;
    \]
    \item connected if for all \( x, y \in X \) there exists a path linking \( x \) to \( y \), namely a list of vertices \( \gamma = \{z_1, z_2, \ldots, z_n\} \subseteq X \) such that:
\vspace{-0.35em}
\[
z_1 = x, \hspace{1.4em} z_n = y \hspace{1.4em} \text{and} \hspace{1.4em} z_i \sim z_{i+1} \hspace{1.2em} \text{for} \hspace{0.5em} i = 1, \ldots, n-1.
\]
\end{enumerate}
\end{definition}

\begin{definition}\label{definition_weighted_degree}
Let \( (X, \omega, \mu) \) be a weighted graph. Then we define, for any vertex \( x \in X \), the degree of \( x \) as
\[
\deg(x) := \sum_{y \in X} \omega(x,y)
\]
\noindent and the weighted degree of \( x \) as
\[
\operatorname{Deg}(x) := \frac{\deg(x)}{\mu(x)} = \frac{1}{\mu(x)} \sum_{y \in X} \omega(x,y).
\]
\end{definition}

\medskip
\noindent Notice that condition c) in \cref{definition_weighted_graph} ensures that every vertex in a weighted graph has finite degree.

\medskip
\begin{definition}\label{definition_pseudo_metric}
Given a countable set \( X \), a pseudo metric on \( X \) is a function \( d : X \times X \to [0, +\infty) \) such that
\begin{enumerate}
\renewcommand{\labelenumi}{\alph{enumi})}
  \item \( d(x,x) = 0 \) for all \( x \in X \);
  \item \( d \) is symmetric, namely \( d(x,y) = d(y,x) \) for all \( x, y \in X \);
  \item \( d \) satisfies the triangular inequality, that is:
  \[
  d(x,y) \leq d(x,z) + d(z,y) \hspace{1.6em} \text{for all } x, y, z \in X.
  \]
\end{enumerate}
\end{definition}

\begin{definition}\label{definition_jump_size}
Let \( (X, \omega, \mu) \) be a weighted graph, and consider a pseudo  metric \( d \) on \( X \). Then the jump size of \( d \) is defined as
\[
\begin{aligned}
s :=& \sup\{ d(x,y) : x, y \in X, \text{ } \omega(x,y) > 0 \} \\
=& \sup\{ d(x,y) : x, y \in X, \text{ } x \sim y \}.
\end{aligned}
\]

For all $ r > 0 , x_0\in X$ we define the ball of radius \( r \) centred at \( x_0 \) as
\[
B_r(x_0) := \{ x \in X : d(x,x_0) < r \};
\]
\noindent furthermore, we set
\[
\partial B_r(x_0) := \{ x \in X : d(x,x_0) = r \}.
\]
\end{definition}

\smallskip
\begin{definition}\label{definition_intrinsic}
Let \( (X, \omega, \mu) \) be a weighted graph, and consider a pseudo  metric \( d \) on \( X \). Given two constants \( q \geq 1, C_0 > 0 \), we say that \( d \) is \( q \)--intrinsic with bound \( C_0 \) if the following holds:
\[
\frac{1}{\mu(x)} \sum_{y \in X} \omega(x,y) \left[ d(x,y) \right]^q \leq C_0 \hspace{1.6em} \text{for all } x \in X.
\]
\noindent In particular, we say that \( d \) is intrinsic if the above estimate is satisfied with \( q = 2, C_0 = 1 \), namely:
\[
\frac{1}{\mu(x)} \sum_{y \in X} \omega(x,y) \left[ d(x,y) \right]^2 \leq 1 \hspace{1.6em} \text{for all } x \in X.
\]
\end{definition}


\subsection{Difference and Laplace operators}\label{subsection_difference_Laplace_operators}

\noindent We first introduce the space of all real-valued functions on the vertex set:
\[
C(X) := \{ f : X \to \mathbb{R} \}.
\]

\begin{definition}\label{definition_difference_operators}
Let \( (X, \omega, \mu) \) be a weighted graph. Then:
\begin{enumerate}
\renewcommand{\labelenumi}{\alph{enumi})}
    \item given two vertices \( x,y \in X \), we define the corresponding difference operator as \( \nabla_{xy} \), acting in the following way for any \( f \in C(X) \):
    \[
    \nabla_{xy} f := f(y) - f(x);
    \]

    \item the (weighted) Laplacian acts as \( \Delta: \mathcal{D}_{\Delta}(X) \to C(X) \), where
    \[
    \mathcal{D}_{\Delta}(X) := \left\{ f \in C(X) : \sum_{y \in X} \omega(x,y) |f(y)| < +\infty, \text{ } \forall x \in X \right\},
    \]
    \noindent and is defined, for all \( x \in X \), as:
    \[
    \begin{aligned}
    \Delta f(x) :=& \frac{1}{\mu(x)} \sum_{y \in X} \omega(x,y) \left[ f(y) - f(x) \right] \\
    =& \frac{1}{\mu(x)} \sum_{y \in X} \omega(x,y) \left[ \nabla_{xy} f \right].
    \end{aligned}
    \]
\end{enumerate}
\end{definition}

\medskip
\noindent Notice that, since by assumption \( \mu > 0 \) on every vertex, then the Laplacian of a function belonging to \( \mathcal{D}_{\Delta}(X) \) is defined over the whole vertex set \( X \). 

\begin{remark}\label{remark_locally_finite_laplacian}
It is straightforward to verify that, for any locally finite weighted graph \( (X, \omega, \mu) \), the following identity holds:
\[
C(X) = \mathcal{D}_{\Delta}(X),
\]
\noindent so that the Laplacian is well defined for any function \( f \in C(X) \).
\end{remark}

\medskip
\noindent Let us define, for any \( f \in C(X) \), the \emph{support} of \( f \) as the set of vertices on which \( f \) does not vanish, namely:
\[
\operatorname{supp}(f) := \left\{ x \in X : f(x) \neq 0 \right\}.
\]
\noindent We also denote by \( C_c(X) \) the set of functions with \emph{finite support}, that is:
\[
C_c(X) := \left\{ f \in C(X) : \operatorname{supp}(f) \text{ is a finite set} \right\}.
\]

\medskip
\noindent We now recall the following identities, which will be expedient in the sequel.

\smallskip
\begin{proposition}
Let \( (X, \omega, \mu) \) be a locally finite weighted graph. Then
\begin{enumerate}
\renewcommand{\labelenumi}{\alph{enumi})}
    \item the following product rule holds:
    \[
    \nabla_{xy}(fg) = f(x) \left[ \nabla_{xy}g \right] + \left[ \nabla_{xy}f \right] g(y) \hspace{1.6em} \text{for all } x, y \in X,
    \]
    \noindent for every couple of functions \( f,g \in C(X) \);

    \item the following formula for the Laplacian of the product holds:
    \begin{equation}\label{Laplacian_product_formula}
    \Delta(fg)(x) = f(x) \Delta g(x) + g(x) \Delta f(x) + \frac{1}{\mu(x)} \sum_{y \in X} \omega(x,y) \bigl[ \nabla_{xy}f \bigr] \left[ \nabla_{xy}g \right],
    \end{equation}
    \noindent for all \( x \in X \) and for every couple of functions \( f,g \in C(X) \);

    \item the following integration by parts formula holds:
    \begin{equation}\label{integration_by_parts}
    \begin{aligned}
    \sum_{x \in X} [\Delta f(x)] g(x) \mu(x) &= -\frac{1}{2} \sum_{x,y \in X} \omega(x,y) \left[ \nabla_{xy} f \right] \left[ \nabla_{xy} g \right] \\
    &= \sum_{x \in X} [\Delta g(x)] f(x) \mu(x),
    \end{aligned}
    \end{equation}
    \noindent for every couple of functions \( f, g \in C(X) \) such that at least one between \( f \) and \( g \) belongs to \( C_c(X) \).
\end{enumerate}
\end{proposition}


\subsection{The combinatorial distance, inner and outer degrees}\label{subsection_inner_outer_degrees}

\noindent Throughout the following, we shall adopt the notation \( \mathbb{N}_0 \) to denote the set \( \mathbb{N} \cup \{ 0 \} \).

\begin{definition}\label{definition_combinatorial_distance}
Let \( (X,\omega,\mu) \) be a connected weighted graph. We define the combinatorial graph distance \( \rho : X \times X \to \mathbb{N}_0 \) as follows: for any pair of vertices \( x, y \in X \), \( \rho(x,y) \) denotes the minimal number of edges in a path connecting \( x \) to \( y \). In other words, we set:
\[
\rho(x,y) := \inf \left\{ n : \gamma = \{x_k\}_{k=0}^n \text{ is a path between } x \text{ and } y \right\}, \hspace{1.8em} \text{for all } x,y \in X.
\]
\noindent Furthermore, given a nonempty finite subset \( \Omega \subset X \), we define the distance from any vertex of the graph to \( \Omega \) as
\[
\rho(x,\Omega) := \min_{y \in \Omega} \hspace{0.1em} \rho(x,y) \hspace{1.8em} \text{for all } x \in X.
\]
\noindent Within this framework, for any \( r \in \mathbb{N}_0 \), we refer to the shell of radius \( r \) centered at \( \Omega \) as
\[
S_r(\Omega) := \left\{ x \in X : \rho(x, \Omega) = r \right\}.
\]
\end{definition}

\noindent We point out a mild abuse of notation in the preceding definition: for a fixed vertex \( x \in X \), we write \( \rho(x,\cdot) \) both for the distance from \( x \) to another vertex of the graph and for the distance from \( x \) to the finite set \( \Omega \subset X \). \noindent The meaning will always be clear from the context.

\medskip
\noindent We also observe that the combinatorial distance is well defined on \( X \times X \), since the weighted graph \( (X,\omega,\mu) \) is assumed to be connected. Moreover, it is straightforward to verify that \( \rho \) actually defines a metric on \( X \).

\medskip
\noindent It is immediate to see that   for all \( x \in X \)
\begin{equation}\label{equivalence_S_0}
\rho(x,\Omega) = 0 \hspace{0.3em} \Leftrightarrow \hspace{0.3em} x \in \Omega.
\end{equation}
This is equivalent to writing
\[
\Omega = S_0(\Omega).
\]

\noindent We now proceed by stating a few elementary results that will be useful later on.

\begin{lemma}\label{lemma_partition_implications_rho}
Consider the framework and the notations introduced in \emph{\cref{definition_combinatorial_distance}}. Then:
\begin{enumerate}
\renewcommand{\labelenumi}{\alph{enumi})}
\item the following identity holds:
\[
X = \bigcup_{r=0}^{+\infty} S_r(\Omega);
\]

\item if the weighted graph \( (X,\omega,\mu) \) is also locally finite, then for each \( r \in \mathbb{N}_0 \) the shell \( S_r(\Omega) \) is nonempty and finite.
\end{enumerate}
\end{lemma}

\begin{remark}\label{remark_partition_X_shells}
\noindent Notice that, if \( (X,\omega,\mu) \) is connected and locally finite, then the family \( \{ S_r(\Omega) \}_{r=0}^{+\infty} \) is a partition of the vertex set \( X \). This follows directly from the identity a) in \emph{\cref{lemma_partition_implications_rho}}, and from the fact that the shells are both nonempty, by property b) of \emph{\cref{lemma_partition_implications_rho}}, and pairwise disjoint, by definition.
\end{remark}

\begin{definition}\label{definition_inner_outer_degrees}
Consider the framework and the notations introduced in \emph{\cref{definition_combinatorial_distance}}, and set
\[
r(x) := \rho(x,\Omega) \hspace{1.8em} \text{for all } x \in X.
\]
\noindent Then, for any \( x \in X \), we define the two functions \( \mathfrak{D}_-, \hspace{0.05em} \mathfrak{D}_+ : X \to [0,+\infty) \) as
\[
\mathfrak{D}_-(x) :=
\begin{cases}
0 & \text{ if } \hspace{0.1em} r(x) = 0 \\
\frac{1}{\mu(x)} \hspace{0.1em} \sum \limits_{y \in S_{r(x)-1}(\Omega)} \omega(x,y) & \text{ if } r(x) \geq 1
\end{cases}
\]
\noindent and
\[
\mathfrak{D}_+(x) := \frac{1}{\mu(x)} \sum \limits_{y \in S_{r(x)+1}(\Omega)} \omega(x,y), \hspace{1.8em} \text{for all } x \in X.
\]
\noindent We refer to \( \mathfrak{D}_- \) and \( \mathfrak{D}_+ \), respectively, as the inner and the outer degree with respect to \( \Omega \).
\end{definition}

\medskip
\noindent We observe that, thanks to \eqref{equivalence_S_0}, the definition of inner degree can be expressed in terms of \( \Omega \), namely:
\[
\mathfrak{D}_-(x) =
\begin{cases}
0 & \text{ if } \hspace{0.1em} x \in \Omega \\
\frac{1}{\mu(x)} \hspace{0.1em} \sum \limits_{y \in S_{r(x)-1}(\Omega)} \omega(x,y) & \text{ if } x \in X \setminus \Omega.
\end{cases}
\]


\subsection{Spherically symmetric functions}\label{subsection_spherically_symmetric_functions}

\begin{definition}\label{definition_spherically_symmetric_function}
We say that a function \( f \in C(X) \) is spherically symmetric with respect to \( \Omega \) if the following condition holds:
\[
f(x) = f(y) \hspace{1.8em} \text{for all } x, y \in X \text{ such that } \rho(x,\Omega) = \rho(y,\Omega).
\]
\noindent In this case, the values taken by \( f \) depend solely on the distance of the vertices from the finite subset \( \Omega \), and we may therefore write, for any \( r \in \mathbb{N}_0 \):
\[
f(x) = f(r) \hspace{1.8em} \text{for all } x \in S_r(\Omega).
\]
\end{definition}

\medskip
\noindent We remark that the final equality in the previous definition involves an abuse of notation, as we use the same symbol \( f \) to denote both the function defined on the vertex set \( X \) and the one defined on \( \mathbb{N}_0 \). Nevertheless, the meaning will always be clear from the context.

\medskip
\begin{definition}\label{definition_weakly_spherically_symmetric}
We say that the weighted graph \( (X, \omega, \mu) \) is weakly spherically symmetric with respect to \( \Omega \) if both the inner degree \( \mathfrak{D}_- \) and the outer degree \( \mathfrak{D}_+ \) are spherically symmetric with respect to \( \Omega \). In this case, in accordance with \emph{\cref{definition_spherically_symmetric_function}}, and with a slight abuse of notation, we write, for any \( r \in \mathbb{N}_0 \):
\[
\mathfrak{D}_\pm(x) = \mathfrak{D}_\pm(r) \hspace{1.8em} \text{for all } x \in S_r(\Omega).
\]
\end{definition}

\medskip
\noindent We also recall the following result (see, e.g., \cite{BP}).

\medskip
\begin{lemma}\label{lemma_formula_Laplacian_for_WSS_graphs}
Let \( (X, \omega, \mu) \) be a connected, locally finite weighted graph, and consider a nonempty finite subset \( \Omega \subset X \). Moreover, let \( f \in C(X) \) be a spherically symmetric function with respect to \( \Omega \). Then, for any \( x \in X \setminus \Omega \), the following identity holds:
\[
\Delta f(x) = \mathfrak{D}_+(x) \left[ f(r(x)+1) - f(r(x)) \right] + \mathfrak{D}_-(x) \left[ f(r(x)-1) - f(r(x)) \right].
\]
\noindent On the other hand, for all \( x \in \Omega \) it holds:
\[
\Delta f(x) = \mathfrak{D}_+(x) \left[ f(1) - f(0) \right].
\]
\end{lemma}



\subsection{Functional spaces}\label{subsection_functional_spaces}

\noindent For any \( p \in [1,+\infty] \), we define the weighted sequence spaces as
\[
\begin{split}
\ell^p(X, \mu) := & \left\{ f \in C(X) : \sum_{x \in X} |f(x)|^p \, \mu(x) < +\infty \right\}, \quad \text{whenever } p \in [1,+\infty), \\
\ell^{\infty}(X, \mu) \equiv & \, \, \ell^{\infty}(X) := \left\{ f \in C(X) : \sup_{x \in X} |f(x)| < +\infty \right\},
\end{split}
\]
\noindent and equip them with the following norms:
\[
\begin{split}
&\|f\|_p := \left[ \sum_{x \in X} |f(x)|^p \, \mu(x) \right]^{1/p} \qquad \text{for all } p \in [1,+\infty), \\
&\|f\|_{\infty} := \sup_{x \in X} |f(x)|.
\end{split}
\]

%


\subsection{Model trees}\label{subsection_model_trees}

\begin{definition}\label{definition_model_tree}
Consider a connected weighted graph \( (\mathbb{T},\omega_0,\mu_1) \) and a reference vertex \( x_0 \in \mathbb{T} \), and define the shells of radius \( r \in \mathbb{N}_0 \) as
\[
S_r(x_0) := \left\{ x \in \mathbb{T} : \rho(x,x_0) = r \right\},
\]
\noindent where \( \rho \) denotes the combinatorial distance, introduced in \emph{\cref{definition_combinatorial_distance}}. Assume that the following properties hold:
\begin{enumerate}
\renewcommand{\labelenumi}{\alph{enumi})}
    \item the edge weight \( \omega_0 : \mathbb{T} \times \mathbb{T} \to \{0,1\} \) is defined as
    \begin{equation}\label{definition_omega_0}
    \omega_0(x,y) :=
    \begin{cases}
    1 & \text{ if } x \sim y \\
    0 & \text{ otherwise;}
    \end{cases}
    \end{equation}
    \item there are no edges connecting vertices within the same shell, i.e., for any \( r \in \mathbb{N}_0 \) it holds:
    \[
    \omega_0(x,y) = 0 \hspace{1.8em} \text{for all } (x,y) \in S_r(x_0) \times S_r(x_0);
    \]
    \item \( \mu_1(x) \equiv 1 \) for every \( x \in \mathbb{T} \);
    \item the inner degree with respect to the set \( \left\{ x_0 \right\} \) has the following expression:
    \[
    \begin{aligned}
    \mathfrak{D}_-(x) & =
    \begin{cases}
    0 & \text{ if } \hspace{0.1em} x = x_0 \\
    \left| \{ y \in \mathbb{T} : y \sim x, \rho(y,x_0) = \rho(x,x_0) - 1 \} \right| & \text{ if } x \in \mathbb{T} \setminus \{ x_0 \}
    \end{cases}
    \\
    & =
    \begin{cases}
    0 & \text{ if } x = x_0 \\
    1 & \text{ if } x \in \mathbb{T} \setminus \{ x_0 \};
    \end{cases}
    \end{aligned}
    \]
    \item the degree function is spherically symmetric with respect to \( \left\{ x_0 \right\} \), that is, \( \deg \in C(\mathbb{T}) \) and it holds:
    \[
    \deg(x) = \deg(y) \hspace{1.8em} \text{for all } x, y \in \mathbb{T} \text{ such that } \rho(x,x_0) = \rho(y,x_0),
    \]
    \noindent hence, with the same abuse of notation introduced in \emph{\cref{definition_spherically_symmetric_function}}, for any \( r \in \mathbb{N}_0 \) we can write
    \[
    \deg(x) = \deg(r) \hspace{1.8em} \text{for all } x \in S_r(x_0).
    \]
\end{enumerate}
\noindent Then \( (\mathbb{T}, \omega_0, \mu_1) \) is said to be a model tree, and the reference vertex \( x_0 \) is referred to as the root of the model.

\medskip
\noindent Finally, we define the branching function \( b : \mathbb{N}_0 \to \mathbb{N} \) associated to the model tree \( (\mathbb{T}, \omega_0, \mu_1) \) as follows:
\[
b(r) :=
\begin{cases}
\deg(x_0) & \text{ if } r = 0 \\
\deg(r) - 1 & \text{ if } r \in \mathbb{N}.
\end{cases}
\]
\noindent If the branching function is constant, namely there exists \( b \in \mathbb{N} \) for which
\[
b(r) = b \hspace{1.6em} \text{for all } r \in \mathbb{N}_0,
\]
\noindent then the model tree is said to be homogeneous, and we denote it as \( (\mathbb{T}_b,\omega_0,\mu_1) \).
\end{definition}

\noindent We highlight that, in practical terms, the branching value \( b(r) \) represents the number of neighbors of any vertex in \( S_r(x_0) \) belonging to \( S_{r+1}(x_0) \), or equivalently, the number of edges connecting each vertex in \( S_r(x_0) \) to a vertex in \( S_{r+1}(x_0) \).

We set
\[
r(x) := \rho(x,x_0), \hspace{1.8em} \text{for all } x \in \mathbb{T}.
\]

\noindent It is direct to see that, for every vertex \(x \neq x_0\), the quantities \(\mathfrak{D}_-(x)\) and \(\mathfrak{D}_+(x)\) represent the number of neighbors of \(x\) lying, respectively, in the preceding shell and in the following shell relative to the one containing \(x\).

\medskip
\noindent We now state, for future reference, the following elementary results.

\begin{lemma}\label{lemma_impliactions_model_trees}
Let \( (\mathbb{T}, \omega_0, \mu_1) \) be a model tree with root \( \{ x_0 \} \). Then:
\begin{enumerate}
\renewcommand{\labelenumi}{\alph{enumi})}
\item it holds
\[
\omega_0(x,y) = 1 \quad \Rightarrow \quad y \in S_{r(x)-1}(x_0) \cup S_{r(x)+1}(x_0),
\]
\noindent for all \( x \in \mathbb{T} \setminus \{ x_0 \} \) and \( y \in \mathbb{T} \), or equivalently
\[
\left\{ y \in \mathbb{T} : x \sim y \right\} \subseteq S_{r(x)-1}(x_0) \cup S_{r(x)+1}(x_0), \hspace{1.8em} \text{\emph{for all} } x \in \mathbb{T} \setminus \{ x_0 \};
\]

\item it holds
\[
\omega_0(y,x_0) = 1 \quad \Leftrightarrow \quad y \in S_1(x_0),
\]
\noindent for all \( y \in \mathbb{T} \), or equivalently
\begin{equation}\label{property_0_1_second_way_model_trees}
\left\{ y \in \mathbb{T} : y \sim x_0 \right\} = S_1(x_0);
\end{equation}

\item the outer degree with respect to the set \( \{ x_0 \} \) satisfies
\[
\mathfrak{D}_+(x) = b(r(x)) \hspace{1.8em} \text{for all } x \in \mathbb{T};
\]

\item \( (\mathbb{T},\omega_0,\mu_1) \) is weakly spherically symmetric with respect to \( \{ x_0 \} \).
\end{enumerate}
\end{lemma}

\begin{lemma}\label{lemma_properties_homogeneous_trees}
Let \( (\mathbb{T}_b,\omega_0,\mu_1) \) be a homogeneous model tree with root \( x_0 \) and branching \( b \in \mathbb{N} \). Then:
\begin{enumerate}
\renewcommand{\labelenumi}{\alph{enumi})}
    \item the inner degree with respect to the set \( \left\{ x_0 \right\} \) satisfies the following identities:
    \[
    \begin{split}
    & \mathfrak{D}_-(x_0) = 0, \\
    & \mathfrak{D}_-(x) = \left| \{ y \in \mathbb{T}_b : y \sim x, \rho(y,x_0) = \rho(x,x_0) - 1 \} \right| = 1 \hspace{1.8em} \text{for all } x \in \mathbb{T}_b \setminus \{ x_0 \};
    \end{split}
    \]

    \item the outer degree with respect to the set \( \left\{ x_0 \right\} \) satisfies the following identities:
    \[
    \mathfrak{D}_+(x) = \left| \{ y \in \mathbb{T}_b : y \sim x, \rho(y,x_0) = \rho(x,x_0) + 1 \} \right| = b \hspace{1.8em} \text{for all } x \in \mathbb{T}_b;
    \]

    \item \( (\mathbb{T}_b,\omega_0,\mu_1) \) is weakly spherically symmetric with respect to \( \{ x_0 \} \).
\end{enumerate}
\end{lemma}


\subsection{The integer lattice}\label{subsection_integer_lattice}

\begin{definition}\label{definition_lattice}
Let \( N \in \mathbb{N} \), and consider a connected weighted graph fulfilling the following properties:
\begin{enumerate}
\renewcommand{\labelenumi}{\alph{enumi})}
    \item the vertex set coincides with \( \mathbb{Z}^N \), that is, each vertex \( x \) is composed of \( N \) integer components:
    \[
    x = (x_1, x_2, \dots, x_N), \hspace{1.8em} \text{with } x_i \in \mathbb{Z} \text{ for each } i \in \left\{ 1, 2, \dots, N \right\};
    \]
    \item for each couple of vertices \( x,y \in \mathbb{Z}^N \), we have \( x \sim y \) if and only if there exists \( k \in \left\{ 1, 2, \dots, N \right\} \) such that
    \[
    y_k = x_k \pm 1 \hspace{1.2em} \text{and} \hspace{1.2em} y_i = x_i  \hspace{0.7em} \text{for each } i \in \left\{ 1, 2, \dots, N \right\} \setminus \{ k \};
    \]
    \item the edge weight \( \omega_0 \) satisfies the definition given in \eqref{definition_omega_0}, namely \( \omega_0 : \mathbb{Z}^N \times \mathbb{Z}^N \to \{0,1\} \) and
    \[
    \omega_0(x,y) :=
    \begin{cases}
    1 & \text{ if } x \sim y \\
    0 & \text{ otherwise;}
    \end{cases}
    \]
    \item the node measure is defined as
    \[
    \mu(x) \equiv 2N \hspace{1.8em} \text{for every } x \in \mathbb{Z}^N.
    \]
\end{enumerate}
\noindent Then the corresponding weighted graph \( (\mathbb{Z}^N,\omega_0,\mu) \) is referred to as the  \( N \)--dimensional integer lattice.
\end{definition}

\medskip
\begin{remark}\label{remark_lattice_N_1_WSS}
\noindent It can be proved that, if \( N \geq 2 \), then the integer lattice \( (\mathbb{Z}^N,\omega_0,\mu) \), equipped with the combinatorial metric \( \rho \), is not weakly spherically symmetric with respect to any nonempty, finite subset of \( \mathbb{Z}^N \). 
These considerations imply that, in the context of the integer lattice, it is not possible to apply the theory of weak spherical symmetry in every spatial dimension \( N \in \mathbb{N} \); for this reasons, it is natural to consider a different choice of distance. In particular, since \( \mathbb{Z}^N \) is a discrete subset of \( \mathbb{R}^N \), it is convenient to equip the lattice with the Euclidean distance.
\end{remark}

\begin{definition}\label{definition_euclidean_distance_lattice}
\noindent We equip the integer lattice \( (\mathbb{Z}^N,\omega_0,\mu) \) with the Euclidean distance \( d : \mathbb{Z}^N \times \mathbb{Z}^N \to [0,+\infty) \), defined as follows:
\[
d(x,y) := \left[ \sum_{i=1}^N |x_i - y_i|^2 \right]^{\frac{1}{2}} \hspace{1.8em} \text{for all } x,y \in \mathbb{Z}^N,
\]
\noindent for which we will also use the notation \( \left| x-y \right| \).
\end{definition}

\noindent For future reference, we characterize adjacency between two arbitrary vertices of the lattice.

\begin{lemma}\label{lemma_characterization_neighbors_lattice}
\noindent Consider the integer lattice graph \( (\mathbb{Z}^N,\omega_0,\mu) \), and let \( x,y \in \mathbb{Z}^N \). Then the following conditions are equivalent:
\begin{enumerate}
\renewcommand{\labelenumi}{\alph{enumi})}
    \item \( x \sim y \);
    \item \( d(x,y) = 1 \).
\end{enumerate}
\end{lemma}



\section{Hypotheses and definition of solution}\label{section_hypotheses}

\begin{definition}\label{definition_PM}
Let \( (X, \omega, \mu) \) be a weighted graph and consider a pseudo metric \( d \) on \( X \). We say that \( d \) satisfies the property  \emph{(PM)} if the following conditions hold:
\begin{enumerate}
\renewcommand{\labelenumi}{\alph{enumi})}
  \item the jump size associated with \( d \) is positive and finite, namely
  \[
  0 < s = \sup\{ d(x,y) : x, y \in X, \text{ } x \sim y \} < +\infty;
  \]
  \item for any fixed \( x \in X \), the ball \( B_r(x) \) is a finite set for all \( r > 0 \);
  \item there exists \( C_0 > 0 \) such that \( d \) is 1--intrinsic with bound \( C_0 \).
\end{enumerate}
\end{definition}

\begin{assumption}\label{assumption_graph_pseudo_metric}
We suppose that \( (X, \omega, \mu) \) is an infinite, connected, weighted graph. Moreover, we assume that there exists a pseudo metric \( d \) on \( X \) satisfying property \emph{(PM)}, according to \emph{\cref{definition_PM}}.
\end{assumption}

\begin{remark}\label{remark_assumptions_G}
\noindent Observe that, under \emph{\cref{assumption_graph_pseudo_metric}}, the weighted graph \( (X, \omega, \mu) \) is necessarily locally finite. In fact, by the definition of the jump size \( s \in (0,+\infty) \), we have \( d(x,y) \leq s \) for every couple of neighbors \( y \sim x \), so that
\[
\{ y \in X : y \sim x \} \subseteq \overline{B}_s(x).
\]
\noindent Since the balls are finite by assumption, it follows that the set of neighbors of each vertex \( x \in X \) is finite.
\end{remark}

\begin{assumption}\label{hypf} We assume that $f:[0,\infty)\to[0,\infty)$ is a locally Lipschitz and convex function, with
$f(0)=0$ and $f(u)>0$ for all $u>0$.
Furthermore,
\begin{equation}\label{eq:osgood}
\int_{1}^{\infty}\frac{du}{f(u)}<\infty
\qquad\text{and}\qquad
\lim_{u\to\infty}\frac{f(u)}{u}=+\infty .
\end{equation}
\end{assumption}

\noindent In addition, we assume that the initial datum \( u_0 \) of problem \eqref{problem_HR_G} is nonnegative and bounded, namely:
\begin{equation}\label{eq:HP_ID_G}
u_0 : X \to [0,+\infty) \hspace{1.6em} \text{and} \hspace{1.6em} u_0 \in \ell^\infty(X).
\end{equation}

\medskip
\noindent \noindent Now that the working assumptions on the graph and its geometry have been clarified, we are ready to introduce the notion of solution for problem \eqref{problem_HR_G}. In particular, we restrict our attention to nonnegative classical solutions that remain bounded for all times prior to \( T \).

\medskip
Let $S_T := X \times (0, T).$

\begin{definition}\label{definition_sub_super_sols_G}
Let \( u_0 \) satisfy \eqref{eq:HP_ID_G} and consider a function \( u : X \times [0,T) \to \mathbb{R} \) such that
\begin{equation}\label{reg_sol_1_G}
  u(x,\cdot) \in C^1((0,T)) \cap C^0([0,T)) \hspace{1.6em} \text{for all } x \in X,
\end{equation}

\vspace{-1.5em}

\begin{equation}\label{reg_sol_2_G}
  u \in L^\infty(X \times [0,T']) \quad \text{ for all } \text{ } T' \in (0,T).
\end{equation}
\noindent Then \( u \) is said to be a solution of problem \eqref{problem_HR_G}, with initial condition \( u_0 \), if \( u \) satisfies \eqref{problem_HR_G} pointwise, and \( u \geq 0 \) in \( S_T \).
\end{definition}


\medskip
\noindent We explicitly highlight that, thanks to the regularity condition \eqref{reg_sol_2_G}, it is immediate to deduce that any solution \( u \) of problem \eqref{problem_HR_G} satisfies
\[
u(\cdot,t) \in \ell^\infty(X) \hspace{1.6em} \text{for each } t \in (0,T).
\]

\noindent Obviously, two mutually exclusive possibilities arise: either the maximal existence time is \( T = +\infty \), in which case the function \( u(\cdot,t) \) remains bounded in \( X \) for all times \( t > 0 \), or \( T \in (0,+\infty) \), and \( \|u(\cdot,t)\|_\infty \) exhibits a vertical asymptote at time \( T \). This dichotomy motivates the introduction of the following definition.

\begin{definition}\label{blow_up_global_G}
Let \( u \) be a solution of problem \eqref{problem_HR_G}. Then:
\begin{enumerate}
\renewcommand{\labelenumi}{\alph{enumi})}

  \item we say that \( u \) blows up in finite time, or equivalently that \( u \) is a nonglobal solution of problem \eqref{problem_HR_G}, if the corresponding maximal time of existence is finite, i.e., \( 0 < T < +\infty \), with
\[
u(\cdot,t) \in \ell^\infty(X) \hspace{1.2em} \text{for all } \text{ } t \in (0,T) \hspace{1.6em} \text{and} \hspace{1.6em} \lim_{t \to T^-} \|u(\cdot, t)\|_\infty = +\infty;
\]

\vspace{-1.18em}
  \item we say that \( u \) is a global solution of problem \eqref{problem_HR_G} if the corresponding maximal time of existence is \( T = +\infty \) and
\[
u(\cdot, t) \in \ell^\infty(X) \hspace{1.0em} \text{for all } \hspace{0.15em} t > 0.
\]
\end{enumerate}
\end{definition}



\section{Statement of the main results}\label{section_main_results}
Let  \emph{\cref{hypf}} be fulfilled. Define
\begin{equation}\label{defa}
s_0(\lambda):=\inf\Bigl\{u>0:\ \frac{f(u)}{u}>\lambda\Bigr\} \quad (\lambda>0).
\end{equation}
Clearly, $s_0(\lambda)\in [0, \infty).$ It is easily seen that, whenever $f(u)=u^p\;\, (p>1)$,
\[s_0(\lambda)=\lambda^{\frac{1}{p-1}}.\]

\subsection{Blow-up on general graphs}

\noindent To begin, we introduce an important property; in what follows, the functions which satisfy that will be referred to as \emph{barriers}.

\smallskip
\begin{definition}\label{def_TF_G}
We say that a function \( \varphi \in C(X) \) satisfies the property \emph{(\(\mathcal{B}_G\))} if the following conditions hold:
\begin{enumerate}
\renewcommand{\labelenumi}{\alph{enumi})}

\item \( \varphi(x) > 0 \) for all \( x \in X \);

\item \( \| \varphi \|_1 = \sum \limits_{x \in X} \varphi(x) \mu(x) = 1 \);

\item there exists a constant \( C_1 > 0 \) such that
\[
\sum_{\substack{x,y \in X \\ \hspace{1.0em} x \in \overline{B}_R \setminus B_{(1-\delta)R - 2s}}} \omega(x,y) \hspace{0.2em} |\nabla_{xy} \varphi| \hspace{0.1em} \leq \hspace{0.1em} C_1,
\]
\noindent for any \( R > 0, \delta \in (0,1) \);

\item there exists a constant \( \lambda > 0 \) for which \( \varphi \) satisfies the following inequality:
\[
\Delta \varphi(x) + \lambda \varphi(x) \geq 0 \hspace{2.5em} \text{for all } x \in X.
\]
\end{enumerate}
\end{definition}

We now state a general criterion that ensures finite-time blow-up of solutions.
\begin{theorem}\label{theorem_blow_up_chapter_3}
Let \emph{\cref{assumption_graph_pseudo_metric}} and \emph{\cref{hypf}} be in place, and consider a function \( u_0 \not \equiv 0 \) satisfying \eqref{eq:HP_ID_G}. Let \( u \) solve problem \eqref{problem_HR_G}. Furthermore, assume that \( \varphi \) satisfies property \emph{(\(\mathcal{B}_G\))}, with
\begin{equation}\label{assumption_proposition_blow_up_chapter_3}
\varphi \hspace{0.1em} \sup \limits_{t \in (t_0,T')} |u_t(\cdot,t)| \in \ell^1(X,\mu) \hspace{1.5em} \text{for any } \hspace{0.1em} t_0,T' \hspace{0.1em} \text{such that } \hspace{0.1em} 0 < t_0 < T' < T.
\end{equation}
\noindent If
\begin{equation}\label{largeness_u_0}
\sum_{x \in X} u_0(x) \hspace{0.1em} \varphi(x) \hspace{0.1em} \mu(x) > s_0(\lambda),
\end{equation}
then \( u \) blows up in finite time.
\end{theorem}

Let us observe that the validity of hypothesis \eqref{assumption_proposition_blow_up_chapter_3}  will be discussed in Remark \ref{remark_proposition_blow_up_chapter_3} below.

\begin{remark}
We observe that if $s_0(\lambda)=0$, condition \eqref{largeness_u_0} is satisfied for any nontrivial initial datum $u_0 \not\equiv 0$. On the other hand, if $s_0(\lambda)>0$, condition \eqref{largeness_u_0} imposes a size restriction on the initial datum $u_0$.
\end{remark}

\medskip

In the case of a general graph, we are able to establish the following blow-up result, which is obtained by applying the criterion provided by the previous theorem.
\begin{theorem}\label{theorem_blow_up_general}
Let \( (X,\omega,\mu) \) be a weighted graph satisfying \emph{\cref{assumption_graph_pseudo_metric}}, endowed with the combinatorial distance \( \rho \). Let \emph{\cref{hypf}} be
in place. Let \( \Omega \subset X \) be a nonempty finite subset and define the function \( S : \mathbb{N}_0 \to (0,+\infty) \) as
\begin{equation}\label{function_S}
S(r) := \sum_{x \in S_r(\Omega)} \mu(x) \hspace{1.8em} \text{for all } \hspace{0.1em} r \in \mathbb{N}_0.
\end{equation}
\noindent Furthermore, suppose that
\begin{enumerate}
\renewcommand{\labelenumi}{\alph{enumi})}
    \item \( \sup \limits_{x \in X} \hspace{0.1em} \mathfrak{D}_+(x) < +\infty \) \hspace{0.05em} and \hspace{0.05em} \( \sup \limits_{x \in X} \hspace{0.1em} \mathfrak{D}_-(x) < +\infty \);
    \item the series
    \[
    \sum_{r = 0}^{+\infty} S(r) \hspace{0.1em} e^{-a \hspace{0.05em} r}
    \]
    \noindent is convergent, for some \( a > 0 \).
\end{enumerate}

Consider also a function \( u_0 \not \equiv 0 \) satisfying \eqref{eq:HP_ID_G}, and let \( u \) solve problem \eqref{problem_HR_G}. Finally, let \( \varphi \in C(X) \) be defined as
\begin{equation}\label{definition_varphi_barrier_chapter 3}
\varphi(x) := C \hspace{0.1em} e^{-a \hspace{0.1em} r(x)} \hspace{1.0em} \text{for all } \hspace{0.1em} x \in X, \hspace{1.0em} \text{where} \hspace{1.0em} C := \left[ \sum_{r = 0}^{+\infty} S(r) \hspace{0.1em} e^{-a \hspace{0.05em} r} \right]^{-1},
\end{equation}
\medskip and assume that
\[
\lambda \geq \sup_{x \in X} \hspace{0.1em} \mathfrak{D}_+(x).
\]
\noindent If \eqref{assumption_proposition_blow_up_chapter_3} is fulfilled and
\[
\sum_{x \in X} u_0(x) \hspace{0.1em} \varphi(x) \hspace{0.1em} \mu(x) > s_0(\lambda),
\]
\noindent then \( u \) blows up in finite time.
\end{theorem}

\subsection{Blow-up on trees}

In the case of trees, by applying Theorem \ref{theorem_blow_up_chapter_3}, we prove the following two results: the first concerns general trees, while the second deals with homogeneous trees.
\begin{theorem}\label{theorem_blow_up_model_trees}
Let  \emph{\cref{hypf}} be fulfilled. Let \( (\mathbb{T},\omega_0,\mu_1) \) be a model tree, with root \( x_0 \in \mathbb{T} \) and branching function \( b \). Suppose that it holds
\begin{equation}\label{assumption_sup_branching_finite}
B := \sup_{r \in \mathbb{N}_0} b(r) < +\infty,
\end{equation}
\noindent and let \( a > \log(B) \). Consider also a function \( u_0 \not \equiv 0 \) satisfying \eqref{eq:HP_ID_G}, that is,
\[
u_0 : \mathbb{T} \to [0,+\infty) \hspace{1.6em} \text{and} \hspace{1.6em} u_0 \in \ell^\infty(\mathbb{T}),
\]
\noindent and let \( u \) solve problem \eqref{problem_HR_G} with $X=\mathbb{T}$. Finally, let \( \varphi \in C(\mathbb{T}) \) be defined as
\begin{equation}\label{definition_varphi_model_trees}
\varphi(x) := C \hspace{0.1em} e^{-a \hspace{0.1em} r(x)} \hspace{1.0em} \text{for all } \hspace{0.1em} x \in \mathbb{T}, \hspace{1.0em} \text{where} \hspace{1.0em} C := \left\{ 1 + \sum_{r = 1}^{+\infty} \left[ \prod_{k=0}^{r-1} b(k) \right] e^{-a \hspace{0.05em} r} \right\}^{-1},
\end{equation}
\medskip and assume that \( \lambda \geq B \).

\medskip
\noindent If
\[
\sum_{x \in \mathbb{T}} u_0(x) \hspace{0.1em} \varphi(x) > s_0(\lambda),
\]
\noindent then \( u \) blows up in finite time.
\end{theorem}

\begin{theorem}\label{theorem_blow_up_homogeneous_trees}
Let  \emph{\cref{hypf}} be fulfilled.
\noindent Let \( (\mathbb{T}_b, \omega_0, \mu_1) \) be a homogeneous tree, with root \( x_0 \in \mathbb{T}_b \) and constant branching \( b \in \mathbb{N} \), and suppose that \( a > \log (b) \). Consider also a function \( u_0 \not \equiv 0 \) satisfying \eqref{eq:HP_ID_G}, namely
\[
u_0 : \mathbb{T}_b \to [0,+\infty) \hspace{1.6em} \text{and} \hspace{1.6em} u_0 \in \ell^\infty(\mathbb{T}_b),
\]
\noindent and let \( u \) solve problem \eqref{problem_HR_G} with $X=\mathbb{T}_b$. Finally, let \( \varphi \in C(\mathbb{T}_b) \) be defined as
\begin{equation}\label{definition_varphi_homogeneous_trees}
\varphi(x) := \left( 1 - b \hspace{0.05em} e^{-a} \right) \hspace{0.1em} e^{-a \hspace{0.1em} r(x)} \hspace{1.8em} \text{for all } x \in \mathbb{T}_b,
\end{equation}
and assume that \( \lambda \geq b \).

\medskip
\noindent If
\[
\sum_{x \in \mathbb{T}_b} u_0(x) \hspace{0.1em} \varphi(x) > s_0(\lambda),
\]
\noindent then \( u \) blows up in finite time.
\end{theorem}

\subsection{Blow-up on the integer lattice}
Finally, again by applying  Theorem \ref{theorem_blow_up_chapter_3}, we establish the following blow-up result on $\mathbb Z^N$, involving the theta function \( \theta : (0,+\infty) \to (0,+\infty) \) defined as
\begin{equation}\label{theta_function}
\theta(s) := \sum_{m \in \mathbb{Z}} e^{- \pi \hspace{0.05em} m^2 \hspace{0.05em} s} \hspace{2.5em} \text{for all } \hspace{0.1em} s > 0.
\end{equation}
We note that, in \cite{PZ2}, further results on $\mathbb Z^N$ will be derived as consequences of this theorem.
\begin{theorem}\label{theorem_blow_up_lattice_1}
Let  \emph{\cref{hypf}} be fulfilled.
Consider a function \( u_0 \) satisfying \eqref{eq:HP_ID_G}, with \( u_0 \not \equiv 0 \). Let \( u \) solve problem \eqref{problem_HR_G} with $X=\mathbb Z^N$.
If, for some $k>0$ and $\lambda \geq 1 - e^{-k}$,
\begin{equation}\label{e20f}
\frac{1}{2N} \left[ \theta \left( \frac{k}{\pi} \right) \right]^{-N} \sum_{x \in \mathbb{Z}^N} e^{-k \hspace{0.05em} |x|^2} \, u_0(x) \, \mu(x) > s_0(\lambda),
\end{equation}
then \( u \) blows up in finite time.

In particular, if $f(u)=u^p\; \, (p>1)$, instead of \eqref{e20f}, we can assume that, for some $k>0,$
\begin{equation}\label{condition_blow_up_G_lemma_1}
\sum_{x \in \mathbb{Z}^N} e^{-k \hspace{0.05em} |x|^2} \hspace{0.1em} u_0(x) \, \mu(x) >
(2N)^{\frac{p}{p-1}} \left[ \theta \left( \frac{k}{\pi} \right) \right]^N k^{\frac{1}{p-1}}.
\end{equation}

\end{theorem}





\section{Proof of \cref{theorem_blow_up_chapter_3}}\label{subsection_proofs_1}

\noindent We first set the notation for the indicator function: given any set \( A \subseteq X \), we define \( \mathds{1}_A : X \to \{0,1\} \) as
\[
\mathds{1}_A(x) :=
\begin{cases}
1 & \text{ if } x \in A \\
0 & \text{ if } x \notin A.
\end{cases}
\]

\noindent Moreover, for any \( f \in C(X) \) we denote its \emph{positive part} by \( [f]_+ \in C(X) \), defined as
\[
[f]_+(x) := \max \left\{ f(x),0 \right\} \hspace{1.6em} \text{for all } x \in X.
\]

\noindent We shall establish some basic properties about \emph{cut-off functions}.


\begin{lemma}\label{lemma_COF_G}
Consider a weighted graph \( (X, \omega, \mu) \) and a pseudo metric \( d \) on \( X \) with jump size \( s \), and suppose that \emph{\cref{assumption_graph_pseudo_metric}} is satisfied. Let \( x_0 \in X \) be a fixed reference node, and denote with \( B_r \) the ball with radius \( r > 0 \) centred at \( x_0 \). Consider the cut-off function
\begin{equation}\label{COF_G}
\chi_R(x) := \min \left\{ \frac{\left[ R - s - d(x,x_0) \right]_+}{\delta R}, 1 \right\} \hspace{1.6em} \text{for all } x \in X,
\end{equation}
\noindent where \( R > 0, \delta \in (0,1) \). Then the following properties hold:

\begin{enumerate}
\renewcommand{\labelenumi}{\alph{enumi})}
    \item \( 0 \leq \chi_R(x) \leq 1 \), for all \( x \in X \);
    \item \( \chi_R \) has finite support, namely \( \chi_R \in C_c(X) \);
    \item for all fixed \( x \in X \), it holds \( \lim \limits_{R \to +\infty} \chi_R(x) = 1 \);
    \item \( \chi_R \) satisfies the following bound for the difference operator:
    \begin{equation}\label{estimate_nabla_chi_R_G}
    |\nabla_{xy} \chi_R| \leq \frac{d(x,y)}{\delta R} \hspace{0.2em} \mathds{1}_{\overline{B}_R \setminus B_{(1-\delta)R - 2s}}(x) \hspace{1.6em} \text{for all } x,y \in X \text{ such that } x \sim y;
    \end{equation}
    \item \( \chi_R \) satisfies the following bound for the Laplacian:
    \begin{equation}\label{estimate_Laplacian_chi_R_G}
    |\Delta \chi_R(x)| \leq \frac{C_0}{\delta R} \hspace{0.2em} \mathds{1}_{\overline{B}_R \setminus B_{(1-\delta)R - 2s}}(x) \hspace{1.6em} \text{for all } x \in X.
    \end{equation}
\end{enumerate}
\end{lemma}

\noindent \emph{Proof.} Let \( R > 0, \delta \in (0,1) \) and consider the function \( \chi_R \) defined in \eqref{COF_G}.

\medskip
\noindent a) By the definition of the positive part of a function in \( C(X) \), it follows that \( \chi_R \) is given as the minimum of two nonnegative functions. Therefore, we immediately deduce that \( \chi_R \geq 0 \) in \( X \). Furthermore, the bound \( \chi_R \leq 1 \) in \( X \) is a direct consequence of \eqref{COF_G}.

\medskip
\noindent b) Since \( \chi_R \) maps \( X \) into the compact interval \( [0,1] \), then obviously \( \chi_R \in C(X) \). To prove that \( \chi_R \in C_c(X) \), we distinguish between two cases, corresponding to different ranges of \( R \).

\medskip
\noindent \textbf{(I)} Assuming \( 0 < R \leq s \), we have \( (1 - \delta)R - s < R - s \leq 0 \), yielding \( R - s - d(\cdot,x_0) \leq 0 \) in \( X \). Consequently, it holds \( \chi_R = \min \left\{ 0,1 \right\} \equiv 0 \), so that
\[
\operatorname{supp}(\chi_R) = \left\{ x \in X : \chi_R(x) \neq 0 \right\} = \emptyset,
\]
which is, of course, a finite set.

\medskip
\noindent \textbf{(II)} Consider now the case in which \( R > s \). After fixing an arbitrary vertex \( x \in X \), two possibilities arise. Specifically, if \( d(x,x_0) \geq R - s \), we argue as in \textbf{(I)} and conclude that \( \chi_R(x) = \min \left\{ 0,1 \right\} = 0 \). Conversely, if \( d(x,x_0) < R - s \), then \eqref{COF_G} implies that \( \chi_R(x) > 0 \), since it is given by the minimum of two positive quantities. From these considerations, we deduce:
\[
\operatorname{supp}(\chi_R) = \left\{ x \in X : \chi_R(x) \neq 0 \right\} = \left\{ x \in X : d(x,x_0) < R - s \right\} = B_{R - s},
\]
\noindent and the ball of radius \( R - s > 0 \) centered at \( x_0 \) is finite by \cref{assumption_graph_pseudo_metric}.

\medskip
\noindent By combining the results in \textbf{(I)} and \textbf{(II)}, we conclude that \( \chi_R \) has finite support for any value of \( R > 0 \).

\medskip
\noindent c) Let us consider an arbitrarily fixed vertex \( x \in X \), and set
\[
R_0 := \frac{d(x,x_0) + s}{1 - \delta} > 0.
\]
\noindent If we choose \( R > R_0 \), then we obtain \( 0 \leq d(x,x_0) < (1 - \delta)R - s < R - s \). Hence we compute
\[
\chi_R(x) = \min \left\{ \frac{\left[ R - s - d(x,x_0) \right]_+}{\delta R}, 1 \right\} = \min \left\{ \frac{R - s - d(x,x_0)}{\delta R}, 1 \right\} = 1.
\]
\noindent Therefore, we have \( \chi_R(x) = 1 \) for all \( R > R_0 \), and the desired limit follows.

\medskip
\noindent d) Let us fix two vertices \( x, y \in X \) such that \( x \sim y \). In order to prove the desired inequality, we distinguish three possible cases based on the value of \( d(x, x_0) \in [0, +\infty) \).

\medskip
\noindent \textbf{(A)} First, assume \( x \in B_{(1-\delta)R - 2s} \), so that \( d(x,x_0) < (1 - \delta)R - 2s < (1 - \delta)R - s \). In particular, this implies that
\[
R - s - d(x,x_0) > \delta R > 0,
\]
\noindent and therefore, arguing as in c), we conclude \( \chi_R(x) = 1 \). Moreover, using the triangular inequality and the definition of the jump size \( s \in (0,+\infty) \), we deduce
\[
d(y,x_0) \leq d(x,y) + d(x,x_0) \leq s + d(x,x_0) < s + (1 - \delta)R - 2s = (1 - \delta)R - s.
\]
\noindent Again, by reasoning as in c) we conclude that also \( \chi_R(y) = 1 \), and therefore \( |\nabla_{xy} \chi_R| = 0 \).

\medskip
\noindent \textbf{(B)} Now, suppose that \( x \in X \setminus \overline{B}_R \). In this case we have \( d(x,x_0) > R > R - s \), which immediately yields
\[
\chi_R(x) = \min \left\{ \frac{\left[ R - s - d(x,x_0) \right]_+}{\delta R}, 1 \right\} = \min \left\{ 0, 1 \right\} = 0.
\]
\noindent Moreover, by applying the triangular inequality and recalling the definition of the jump size, we obtain
\[
R < d(x,x_0) \leq d(x,y) + d(y,x_0) \leq s + d(y,x_0),
\]
\noindent hence \( d(y,x_0) > R - s \). Using the same argument as above, we deduce that \( \chi_R(y) = 0 \) as well. Therefore, it holds \( |\nabla_{xy} \chi_R| = 0 \).

\medskip
\noindent \textbf{(C)} Finally, assume \( x \in \overline{B}_R \setminus B_{(1-\delta)R - 2s} \), that is, \( (1-\delta)R - 2s \leq d(x,x_0) \leq R \). We want to show that the following inequality is satisfied:
\begin{equation}\label{case_C_3_G}
|\nabla_{xy} \chi_R| \leq \frac{d(x,y)}{\delta R}.
\end{equation}
\noindent Notice, from \eqref{COF_G}, that \( \chi_R \) can be expressed as the composition of two functions. In particular, after introducing \( h : \mathbb{R} \to [0,1] \) and \( g : X \to \mathbb{R} \), defined as
\[
\begin{split}
\begin{aligned}
h(z) &:= \min \left\{ \left[ z \right]_+, 1 \right\} \hspace{2.78em} \text{for all } z \in \mathbb{R}, \\
g(v) &:= \frac{R - s - d(v,x_0)}{\delta R} \hspace{1.6em} \text{for all } v \in X,
\end{aligned}
\end{split}
\]
\noindent it can be easily seen that \( \chi_R = h \circ g \), namely:
\begin{equation}\label{composition_chi_R_G}
\chi_R(v) = h(g(v)) \hspace{1.6em} \text{for all } v \in X.
\end{equation}
\noindent In the definition of \( h \) given above, the notation \( [z]_+ = \max \{ z, 0 \} \) is used to denote the positive part of a real number \( z \).

\medskip
\noindent We now prove that the function \( h \) is \emph{nonexpansive}, meaning that it satisfies the following inequality:
\begin{equation}\label{non_expansive_1}
|h(b) - h(a)| \leq |b-a| \hspace{1.6em} \text{for all } a,b \in \mathbb{R}.
\end{equation}

\noindent To this end, we may assume without loss of generality that \( a \leq b \) and aim to prove that
\begin{equation}\label{non_expansive_2}
h(b) - h(a) \leq b-a,
\end{equation}
\noindent which is clearly equivalent to \eqref{non_expansive_1}, since \( h \) is nondecreasing on \( \mathbb{R} \). The function \( h \), indeed, can be explicitly expressed as
\[
h(z) =
\begin{cases}
0 & \hspace{1.0em} \text{if } z \leq 0, \\
z & \hspace{1.0em} \text{if } 0 < z < 1, \\
1 & \hspace{1.0em} \text{if } z \geq 1.
\end{cases}
\]
\noindent Proving \eqref{non_expansive_2} under the assumption \( a \leq b \) will thus suffice: once this is done, the general inequality \eqref{non_expansive_1} can be concluded by symmetry, switching the roles of \( a \) and \( b \).

\medskip
\noindent Now, if both \( a \) and \( b \) belong to one of the two regions where \( h \) is constant, namely either \( (-\infty, 0] \) or \( [1,+\infty) \), then \( h(a) = h(b) \), and \eqref{non_expansive_2} clearly holds. In the case where both \( a \) and \( b \) lie in the interval \( (0,1) \), we have \( h(a) = a \) and \( h(b) = b \), so again the desired inequality becomes an identity.

\medskip
\noindent Suppose now that \( a \leq 0 < b \). In this situation we have \( h(a) = 0 \) and \( h(b) = \min\{b,1\} \leq b \), hence:
\[
h(b) - h(a) = h(b) \leq b \leq b - a,
\]
\noindent yielding \eqref{non_expansive_2}. Similarly, in the case \( 0 < a < 1 \leq b \), we compute \( h(a) = a \) and \( h(b) = 1 \), so that
\[
h(b) - h(a) = 1 - a \leq b - a,
\]
\noindent and \eqref{non_expansive_2} holds once again.

\medskip
\noindent All the possible configurations satisfying \( a \leq b \) have thus been checked, and the validity of \eqref{non_expansive_2} has been established in each case. As discussed above, this implies \eqref{non_expansive_1}, confirming that \( h \) is a nonexpansive map.

\medskip
\noindent Therefore, by combining \eqref{composition_chi_R_G}, \eqref{non_expansive_1}, and the definition of \( g \), we obtain:
\[
\begin{aligned}
|\nabla_{xy} \chi_R| &= |\chi_R(y) - \chi_R(x)| \\
&= |h(g(y)) - h(g(x))| \\
&\leq |g(y) - g(x)| \\
&= \left| \frac{R - s - d(y,x_0)}{\delta R} - \frac{R - s - d(x,x_0)}{\delta R} \right| \\
&= \frac{|d(x,x_0) - d(y,x_0)|}{\delta R} \\
&\leq \frac{d(x,y)}{\delta R},
\end{aligned}
\]
\noindent where in the last passage we made use of the triangular inequality. Indeed, by exploiting the facts that \( d(x,x_0) \leq d(x,y) + d(y,x_0) \) and \( d(y,x_0) \leq d(x,y) + d(x,x_0) \) we get:
\[
|d(x,x_0) - d(y,x_0)| \leq d(x,y).
\]
\noindent We have then proved the validity of \eqref{case_C_3_G}.

\medskip
\noindent By collecting the results established in \textbf{(A)}, \textbf{(B)}, and \textbf{(C)}, we see that it holds:
\[
|\nabla_{xy} \chi_R| \leq \frac{d(x,y)}{\delta R} \hspace{0.2em} \mathds{1}_{\overline{B}_R \setminus B_{(1-\delta)R - 2s}}(x).
\]
\noindent Since the choice of the neighboring vertices \( x,y \in X \) was arbitrary, the thesis follows.

\medskip
\noindent e) Let us fix a vertex \( x \in X \) and distinguish two cases depending on the value of \( d(x,x_0) \). First, assume that \( x \in X \setminus \left( \overline{B}_R \setminus B_{(1-\delta)R - 2s} \right) \), meaning that either \( d(x,x_0) < (1 - \delta)R - 2s \) or \( d(x,x_0) > R \). In this case, estimate \eqref{estimate_Laplacian_chi_R_G} simply reduces to \( \Delta \chi_R(x) = 0 \). This follows immediately from \eqref{estimate_nabla_chi_R_G}, which in this scenario reads \( \nabla_{xy} \chi_R = 0 \) for all \( y \sim x \), and therefore
\[
\Delta \chi_R(x) = \frac{1}{\mu(x)} \sum_{y \in X} \omega(x,y) \left[ \nabla_{xy} \chi_R \right] = \frac{1}{\mu(x)} \sum_{\substack{y \in X \\ y \sim x}} \omega(x,y) \left[ \nabla_{xy} \chi_R \right] = 0.
\]
\noindent This proves the validity of \eqref{estimate_Laplacian_chi_R_G}.

\medskip
\noindent Now, suppose that \( x \in \overline{B}_R \setminus B_{(1-\delta)R - 2s} \), namely \( (1 - \delta)R - 2s \leq d(x,x_0) \leq R \). Under this assumption, estimate \eqref{estimate_nabla_chi_R_G} takes the form of \eqref{case_C_3_G}. Among the hypotheses listed in \cref{assumption_graph_pseudo_metric}, we recall that the pseudo metric \( d \) is 1--intrinsic with bound \( C_0 \), which ensures the validity of the following inequality for some constant \( C_0 > 0 \):
\[
\frac{1}{\mu(x)} \sum_{y \in X} \omega(x,y) \hspace{0.1em} d(x,y) \leq C_0.
\]
\noindent We can now exploit this inequality, together with \eqref{case_C_3_G}, and recall that \( \omega(x,y) > 0 \) if and only if \( x \sim y \), in order to obtain:
\[
\begin{aligned}
|\Delta \chi_R(x)| &= \left| \frac{1}{\mu(x)} \sum_{y \in X} \omega(x,y) \left[ \nabla_{xy} \chi_R \right] \right| \\
&\leq \frac{1}{\mu(x)} \sum_{y \in X} \omega(x,y) \left| \nabla_{xy} \chi_R \right| \\
&\leq \frac{1}{\delta R \hspace{0.15em} \mu(x)} \sum_{y \in X} \omega(x,y) \hspace{0.1em} d(x,y) \\
&\leq \frac{C_0}{\delta R},
\end{aligned}
\]
\noindent which proves \eqref{estimate_Laplacian_chi_R_G} in this scenario.

\medskip
\noindent In conclusion, \eqref{estimate_Laplacian_chi_R_G} holds for any arbitrary vertex \( x \in X \), and the thesis follows. \hfill \( \square \)

\begin{remark}
\noindent In relation to the previous proof, observe that once \( \delta \in (0,1) \) is fixed, if the radius \( R \) satisfies
\[
0 < R \leq \frac{2s}{1 - \delta},
\]
\noindent then it follows that \( (1 - \delta)R - 2s \leq 0 \), implying \( B_{(1 - \delta)R - 2s} = \emptyset \). The validity of the proof remains unaffected; however, in the argumentation of points d) and e), one must disregard those sub-cases in which a node is assumed to belong to the ball \( B_{(1 - \delta)R - 2s} \), or alternatively, replace the set \( \overline{B}_R \setminus B_{(1 - \delta)R - 2s} \) by \( \overline{B}_R \), whenever it appears.
\end{remark}

\begin{lemma}\label{lemmaode}
Let  \emph{\cref{hypf}} be fulfilled.
Fix $\lambda>0$ and let $y=y(t)$ satisfy
\begin{equation}\label{eq:ineq}
y'(t)+\lambda y(t)\ge f(y(t)),
\qquad y(0)=y_{0}>0.
\end{equation}
Let $s_0(\lambda)$ be defined by \eqref{defa}. If $$y_{0}>s_0(\lambda),$$
then $y$ blows up in finite time. In particular, there exists
$T^{*}\in(0,\infty)$ such that
$$\lim_{t\to T^{*}}y(t)=+\infty.$$
Moreover, with
\[
c:=1-\frac{\lambda y_{0}}{f(y_{0})}\in(0,1),
\]
one has the explicit upper bound
\begin{equation}\label{eq:time-bound}
T^{*}\le \frac{1}{c}\int_{y_{0}}^{\infty}\frac{du}{f(u)}<\infty .
\end{equation}
In particular, if $f(u)=u^{p}$ with $p>1$, then
$$s_0(\lambda)=\lambda^{1/(p-1)},$$ and
$$T^{*}\le \frac 1{(p-1)(y_{0}^{p-1}-\lambda)}.$$
\end{lemma}

\noindent \emph{Proof.} The conclusion follows from a straightforward qualitative analysis of an ordinary differential equation, which we briefly outline for the reader’s convenience.
Set $g(u):=f(u)-\lambda u$.
Since $f$ is convex and $f(0)=0$, the map $u\mapsto \frac{f(u)}{u}$ is
nondecreasing on $(0,\infty)$.
Hence $y_{0}>s_0(\lambda)$ implies $\frac{f(y_{0})}{y_{0}}>\lambda$, i.e.
$g(y_{0})>0$, and for every $u\ge y_{0}$,
\[
\frac{f(u)}{u}\ge \frac{f(y_{0})}{y_{0}}
\quad\Rightarrow\quad
\frac{\lambda u}{f(u)}\le \frac{\lambda y_{0}}{f(y_{0})}=1-c .
\]
Therefore, for all $u\ge y_{0}$,
\[
g(u)=f(u)\Bigl(1-\frac{\lambda u}{f(u)}\Bigr)\ge c\,f(u)>0 .
\]
From \eqref{eq:ineq} we have $y'(t)\ge g(y(t))$.
Define $\Phi(s):=\int_{y_{0}}^{s}\! \frac{d\xi}{g(\xi)}$ for $s\ge y_{0}$.
Then $\Phi$ is increasing and, by the chain rule, we obtain:
\[
\frac{d}{dt}\Phi(y(t))
=\frac{y'(t)}{g(y(t))}\ge 1,
\]
for all \( t \in (0,T^{*}) \), where \( T^{*} \in (0,+\infty] \) is the maximal time of existence of the function \( y \). Here, the inequality is due to the strict positivity of \( g \). Indeed, from \( y'(t) \ge g(y(t)) \) we deduce that $y'(0) \ge g(y_0) > 0$, and thus $y(t)\ge y_0$ for small times. Now, supposing by contradiction that $y$ drops below $y_0$ at some time, we would infer the existence of a first time $ \tau \in (0,T^{*}) $ such that $y(\tau)=y_0$ and $y(t)\ge y_0$ for all $t\le \tau$. However, in this case it would hold \( y'(\tau) \ge g\big(y(\tau)\big)=g(y_0)>0 \), which contradicts the possibility of decreasing through $y_0$. Therefore, we have \( y(t) \ge y_0 \), and thus $g(y(t))>0$, for every $t\in[0,T^*)$. Now, integrating the previous inequality gives $\Phi(y(t))\ge t$, hence
\[
t \le \int_{y_{0}}^{y(t)} \frac{d\xi}{g(\xi)} \le \int_{y_{0}}^{+\infty} \frac{d\xi}{g(\xi)},
\]
for all $t \in [0, T^{*})$. Letting $t\to T^{*}$ and using $g\ge c f$ yields
\[
T^{*}\le \int_{y_{0}}^{\infty}\frac{du}{g(u)}
\le \frac{1}{c}\int_{y_{0}}^{\infty}\frac{du}{f(u)}<\infty,
\]
where finiteness follows from \eqref{eq:osgood}. We now notice that, if $y$ remained bounded on $[0,T^*)$, then $y(t)$ would stay in a compact interval where $g$ is locally Lipschitz, hence continuous. In this case, the differential inequality would allow us to prolong $y$ beyond $T^*$, contradicting the maximality of $T^*$. Therefore $y(t)$ cannot remain bounded as $t \to T^*$, and we conclude that $y(t)\to+\infty$ as $t\to T^{*}$. Finally, for $f(u)=u^{p}$, $\frac{f(u)}{u}=u^{p-1}$ gives $s_0(\lambda)=\lambda^{1/(p-1)}$, and \eqref{eq:time-bound} gives the stated bound.
\hfill \( \square \)

\bigskip
\noindent \emph{Proof of \emph{\cref{theorem_blow_up_chapter_3}}}. By contradiction, we assume that \( u \) is a global solution to problem \eqref{problem_HR_G}, namely \( T = +\infty \) and \( u(\cdot, t) \in \ell^\infty(X) \) for all \( t > 0 \). Notice that, since by assumption \( u_0 \not \equiv 0 \), then \( u \) cannot coincide with the trivial solution. Moreover, we highlight that \( \varphi u_t(\cdot,t) \in \ell^1(X,\mu) \) for all \( t >  0 \). Indeed, after choosing \( t > 0 \), we can always select \( t_0 \in (0,t) \) and \( T' \in (t,+\infty) \), so that \( 0 < t_0 < t < T' < T = +\infty \). Now, thanks to the assumption \eqref{assumption_proposition_blow_up_chapter_3}, the following estimate holds:
\[
\begin{aligned}
\left| \sum_{x \in X} \varphi(x) \hspace{0.1em} u_t(x,t) \hspace{0.1em} \mu(x) \right| & \leq \sum_{x \in X} \varphi(x) \hspace{0.1em} \left| u_t(x,t) \right| \hspace{0.1em} \mu(x) \\
& \leq \sum_{x \in X} \varphi(x) \hspace{0.1em} \sup \limits_{\tau \in (t_0,T')} |u_t(x,\tau)| \hspace{0.2em} \mu(x) \\
& < +\infty.
\end{aligned}
\]
\noindent We now fix \( t > 0, R > 0, \delta \in (0,1) \) and we multiply the equation in \eqref{problem_HR_G} by \( \chi_R \hspace{0.1em} \varphi \hspace{0.1em} \mu \), where \( \chi_R \) is the cut-off function defined in \eqref{COF_G} as
\[
\chi_R(x) := \min \left\{ \frac{\left[ R - s - d(x,x_0) \right]_+}{\delta R}, 1 \right\} \hspace{1.6em} \text{for all } x \in X.
\]
\noindent This yields:
\begin{equation}\label{first_G}
\begin{aligned}
\sum_{x \in X} \chi_R(x) \, \varphi(x) \, u_t(x,t) \, \mu(x) = & \sum_{x \in X} \chi_R(x) \, \varphi(x) \, \Delta u(x,t) \, \mu(x) \hspace{0.4em} + \\
& + \sum_{x \in X} \chi_R(x) \, \varphi(x) \, f(u(x,t)) \, \mu(x).
\end{aligned}
\end{equation}

\medskip
\noindent Recall from \cref{lemma_COF_G} that \( \chi_R \in C_c(X) \). Moreover, both \( u(\cdot,t) \) and \( \Delta u(\cdot,t) \) belong to \( C(X) \). From \eqref{reg_sol_1_G}, it easily follows that \( u_t(\cdot,t) \in C(X) \) as well. Additionally, the barrier function \( \varphi \) satisfies property (\(\mathcal{B}_G\)), and in particular \( \varphi \in C(X) \). Finally, from \cref{definition_weighted_graph}, we know that \( \mu \in C(X) \). Therefore, all the terms appearing in the sums in \eqref{first_G} belong to \( C_c(X) \), which implies that the sums are on a finite number of terms, hence they are well defined.

\medskip
\noindent We now separately investigate the behavior of the three sums in \eqref{first_G} as \( R \to +\infty \).

\medskip
\noindent \textbf{(I)} By \cref{lemma_COF_G} we know that \( \lim \limits_{R \to +\infty} \chi_R(x) = 1 \), for any fixed \( x \in X \), whence:
\[
\lim_{R \to +\infty} \chi_R(x) \, \varphi(x) \, u_t(x,t) = \varphi(x) \, u_t(x,t) \hspace{1.6em} \text{for all } x \in X.
\]
Moreover, it holds \( \chi_R \leq 1 \) in \( X \), and therefore the following estimate is satisfied for all \( x \in X \) and \( R > 0 \):
\[
|\chi_R(x) \, \varphi(x) \, u_t(x,t)| = \chi_R(x) \, \varphi(x) \, |u_t(x,t)| \leq \varphi(x) \, |u_t(x,t)|,
\]
where we have exploited the fact that both \( \varphi \) and \( \chi_R \) are nonnegative. As already remarked, it holds \( \varphi \hspace{0.1em} u_t(\cdot,t) \in \ell^1(X,\mu) \), hence we can apply the Dominated Convergence Theorem in the context of the measure space \( (X,\mathcal{P}(X),\mu) \), obtaining:
\[
\lim_{R \to +\infty} \sum_{x \in X} \chi_R(x) \, \varphi(x) \, u_t(x,t) \, \mu(x) = \sum_{x \in X} \varphi(x) \, u_t(x,t) \, \mu(x).
\]

\medskip
\noindent \textbf{(II)} By the same arguments used in \textbf{(I)}, we get both
\[
\lim_{R \to +\infty} \chi_R(x) \, \varphi(x) \, f(u(x,t)) = \varphi(x) \, f(u(x,t)) \hspace{1.6em} \text{for all } x \in X
\]
\noindent and
\[
|\chi_R(x) \, \varphi(x) \, f(u(x,t))| \equiv \chi_R(x) \, \varphi(x) \, f(u(x,t)) \leq \varphi(x) \, f(u(x,t)) \leq f(\|u(\cdot,t)\|_\infty) \, \varphi(x),
\]
\noindent for all \( x\in X, R > 0 \). The last inequality holds since the function \( f \) is increasing on \([0,+\infty) \), which is a trivial consequence of the fact that, as already remarked in the proof of \cref{lemmaode}, the map \( u \mapsto \frac{f(u)}{u} \) is nondecreasing on \( (0,+\infty) \). By assumption, \( \varphi \) satisfies property (\(\mathcal{B}_G\)), and in particular \( \varphi \in \ell^1(X,\mu) \). The Dominated Convergence Theorem can then be applied, yielding that \( \varphi f(u(\cdot,t)) \in \ell^1(X,\mu) \) and
\[
\lim_{R \to +\infty} \sum_{x \in X} \chi_R(x) \, \varphi(x) \, f(u(x,t)) \, \mu(x) = \sum_{x \in X} \varphi(x) \, f(u(x,t)) \, \mu(x).
\]

\medskip
\noindent \textbf{(III)} Since \cref{remark_assumptions_G} ensures that, under \cref{assumption_graph_pseudo_metric}, the weighted graph \( (X,\omega,\mu) \) is locally finite, and given that \( u(\cdot, t) \in C(X) \) and \( \chi_R \hspace{0.1em} \varphi \in C_c(X) \), we are allowed to apply formula \eqref{integration_by_parts}, which yields:
\begin{equation}\label{application_IBP_G}
\sum_{x \in X} \chi_R(x) \, \varphi(x) \, \Delta u(x,t) \, \mu(x) = \sum_{x \in X} \Delta \left[ \chi_R \, \varphi \right](x) \, u(x,t) \, \mu(x).
\end{equation}

\noindent Furthermore, by invoking formula \eqref{Laplacian_product_formula}, we obtain for all \( x \in X \):
\[
\Delta \left[ \chi_R \hspace{0.1em} \varphi \right](x) = \chi_R(x) \, \Delta \varphi(x) + \varphi(x) \, \Delta \chi_R(x) + \frac{1}{\mu(x)} \sum_{y \in X} \omega(x,y) \bigl[ \nabla_{xy} \chi_R \bigr] \left[ \nabla_{xy} \varphi \right],
\]
so that \eqref{application_IBP_G} becomes
\[
\begin{aligned}
\sum_{x \in X} \chi_R(x) \, \varphi(x) \, \Delta u(x,t) \, \mu(x) &= \sum_{x \in X} \chi_R(x) \, \Delta \varphi(x) \, u(x,t) \, \mu(x) \hspace{0.4em} + \\
& \hspace{1.1em} + \sum_{x \in X} \Delta \chi_R(x) \, \varphi(x) \, u(x,t) \, \mu(x) \hspace{0.4em} + \\
& \hspace{1.1em} + \sum_{x,y \in X} \omega(x,y) \bigl[ \nabla_{xy} \chi_R \bigr] \left[ \nabla_{xy} \varphi \right] \, u(x,t). \\
\end{aligned}
\]

\noindent Now we make use of property d) from \cref{def_TF_G}, namely the inequality \( \Delta \varphi + \lambda \varphi \geq 0 \), valid in \( X \) for some constant \( \lambda > 0 \). This leads to:
\begin{equation}\label{Laplacian_term_no_limit}
\begin{aligned}
\sum_{x \in X} \chi_R(x) \, \varphi(x) \, \Delta u(x,t) \, \mu(x) &\geq - \lambda \, \sum_{x \in X} \chi_R(x) \, \varphi(x) \, u(x,t) \, \mu(x) \hspace{0.4em} + \\
& \hspace{1.1em} + \sum_{x \in X} \Delta \chi_R(x) \, \varphi(x) \, u(x,t) \, \mu(x) \hspace{0.4em} + \\
& \hspace{1.1em} + \sum_{x,y \in X} \omega(x,y) \bigl[ \nabla_{xy} \chi_R \bigr] \left[ \nabla_{xy} \varphi \right] \, u(x,t). \\
\end{aligned}
\end{equation}
\noindent Again, we separately investigate the behavior of the three sums appearing at the right-hand side of \eqref{Laplacian_term_no_limit}, as \( R \to +\infty \).

\medskip
\noindent \textbf{(A)} Reasoning in the same way as in \textbf{(II)}, we get both
\[
\lim_{R \to +\infty} \chi_R(x) \, \varphi(x) \, u(x,t) = \varphi(x) \, u(x,t) \hspace{1.6em} \text{for all } x \in X
\]
\noindent and
\[
|\chi_R(x) \, \varphi(x) \, u(x,t)| \equiv \chi_R(x) \, \varphi(x) \, u(x,t) \leq \varphi(x) \, u(x,t) \leq \|u(\cdot,t)\|_\infty \, \varphi(x),
\]
\noindent for all \( x \in X, R > 0 \), where we have used the fact that \( u(\cdot,t) \in \ell^\infty(X) \). Since \( \varphi \in \ell^1(X,\mu) \), the Dominated Convergence Theorem allows us to conclude that \( \varphi \hspace{0.1em} u(\cdot,t) \in \ell^1(X,\mu) \) and
\[
\lim_{R \to +\infty} \sum_{x \in X} \chi_R(x) \, \varphi(x) \, u(x,t) \, \mu(x) = \sum_{x \in X} \varphi(x) \, u(x,t) \, \mu(x).
\]

\medskip
\noindent \textbf{(B)} We can write, for any \( R > 0 \):
\[
\begin{aligned}
\left| \sum_{x \in X} \Delta \chi_R(x) \, \varphi(x) \, u(x,t) \, \mu(x) \right| &\leq \sum_{x \in X} \left| \Delta \chi_R(x) \right| \, \varphi(x) \, u(x,t) \, \mu(x) \\
&\leq \|u(\cdot,t)\|_\infty \, \sum_{x \in X} \left| \Delta \chi_R(x) \right| \, \varphi(x) \, \mu(x) \\
&\leq \frac{C_0 \, \|u(\cdot,t)\|_\infty }{\delta R} \, \sum_{x \in \overline{B}_R \setminus B_{(1-\delta)R - 2s}} \varphi(x) \, \mu(x) \\
&\leq \frac{C_0 \, \|u(\cdot,t)\|_\infty }{\delta R} \, \sum_{x \in X} \varphi(x) \, \mu(x) \\
&= \frac{C_0 \, \|u(\cdot,t)\|_\infty }{\delta R}.
\end{aligned}
\]
\noindent In this step, we have first exploited the fact that \( u(\cdot,t) \in \ell^\infty(X) \), followed by the application of \eqref{estimate_Laplacian_chi_R_G}. Finally, the last equality relies on \( \varphi \) having unitary norm in \( \ell^1(X,\mu) \). Since the right-hand side of the resulting inequality vanishes in the limit as \( R \to +\infty \), we conclude:
\[
\lim_{R \to +\infty} \sum_{x \in X} \Delta \chi_R(x) \, \varphi(x) \, u(x,t) \, \mu(x) = 0.
\]

\medskip
\noindent \textbf{(C)} For any \( R > 0 \), the following estimates hold:
\[
\begin{aligned}
\left| \sum_{x,y \in X} \omega(x,y) \bigl[ \nabla_{xy} \chi_R \bigr] \left[ \nabla_{xy} \varphi \right] u(x,t) \right| &\leq
\sum_{x,y \in X} \omega(x,y) \left| \nabla_{xy} \chi_R \right| \, \left| \nabla_{xy} \varphi \right| u(x,t) \\
&\leq \|u(\cdot,t)\|_\infty \, \sum_{x,y \in X} \omega(x,y) \left| \nabla_{xy} \chi_R \right| \, \left| \nabla_{xy} \varphi \right| \\
&\leq \frac{s \, \|u(\cdot,t)\|_\infty}{\delta R} \sum_{\substack{x,y \in X \\ x \in \overline{B}_R \setminus B_{(1-\delta)R - 2s}}} \omega(x,y) \,
\left| \nabla_{xy} \varphi \right| \\
&\leq \frac{C_1 \, s \, \|u(\cdot,t)\|_\infty}{\delta R}.
\end{aligned}
\]
\noindent Indeed, since \( u(\cdot,t) \in \ell^\infty(X) \), the second inequality is true. Moreover, the third one is justified by observing that \( \omega \) is nonzero only on adjacent nodes, together with the application of the estimate
\[
|\nabla_{xy} \chi_R| \leq \frac{s}{\delta R} \, \mathds{1}_{\overline{B}_R \setminus B_{(1-\delta)R - 2s}}(x) \hspace{2.8em} \text{for all } x,y \in X \text{ such that } x \sim y,
\]
\noindent which follows directly from \eqref{estimate_nabla_chi_R_G} and from the definition of the jump size \( s \in (0,+\infty) \). Now, due to property c) in \cref{def_TF_G}, also the last inequality is justified, and the term at the right-hand side of the resulting estimate simply corresponds to a constant multiplied by \( 1/R \). Hence, we conclude:
\[
\lim_{R \to +\infty} \sum_{x,y \in X} \omega(x,y) \bigl[ \nabla_{xy} \chi_R \bigr] \left[ \nabla_{xy} \varphi \right] u(x,t) = 0.
\]

\medskip
\noindent After putting together the results obtained in \textbf{(A)}, \textbf{(B)} and \textbf{(C)}, as \( R \to +\infty \) \eqref{Laplacian_term_no_limit} reads:
\[
\lim_{R \to +\infty} \sum_{x \in X} \chi_R(x) \, \varphi(x) \, \Delta u(x,t) \, \mu(x) \geq - \lambda \sum_{x \in X} \varphi(x) \, u(x,t) \, \mu(x).
\]

\medskip
\noindent Finally, thanks to the result of the computations done in \textbf{(I)}, \textbf{(II)} and \textbf{(III)}, letting \( R \to +\infty \) in \eqref{first_G} yields, for all \( t > 0 \):
\begin{equation}\label{second_G}
\sum_{x \in X} \varphi(x) \, u_t(x,t) \, \mu(x) \geq - \lambda \sum_{x \in X} \varphi(x) \, u(x,t) \, \mu(x) + \sum_{x \in X} \varphi(x) \, f(u(x,t)) \, \mu(x).
\end{equation}

\medskip
\noindent As already noticed, the three series appearing in \eqref{second_G} are convergent for all positive times. This allows us to define the function \( \Phi : [0,+\infty) \to [0,+\infty) \) by
\begin{equation}\label{def_Phi_G}
  \Phi(t) := \sum_{x \in X} \varphi(x) \, u(x,t) \, \mu(x) \hspace{1.6em} \text{for all } t \geq 0,
\end{equation}
\noindent and the initial condition \( u(\cdot,0) = u_0 \) implies
\begin{equation}\label{def_Phi_0_G}
  \Phi(0) = \sum_{x \in X} \varphi(x) \, u_0(x) \, \mu(x) > s_0(\lambda),
\end{equation}
\noindent where the inequality is satisfied by assumption. Observe that the series defining \( \Phi(0) \) also converges: indeed, since \( \varphi \in \ell^1(X,\mu) \) and \( u_0 \in \ell^\infty(X) \) by \eqref{eq:HP_ID_G}, Hölder's inequality ensures that \( \varphi \hspace{0.1em} u_0 \in \ell^1(X,\mu) \).

\medskip
\noindent Now, since the function \( \Phi \) is clearly nonnegative, we have:
\begin{equation}\label{contradiction_1_G}
  0 \leq \Phi(t) < +\infty \hspace{1.8em} \text{for all } t \geq 0.
\end{equation}
\noindent Moreover, \eqref{reg_sol_1_G} implies that, for all \( x \in X \), \( \varphi(x) \hspace{0.1em} u(x,\cdot) \hspace{0.1em} \mu(x) \in C^1((0,+\infty)) \), and therefore, exploiting the assumption \eqref{assumption_proposition_blow_up_chapter_3}, we can apply the theorem of derivation for series, in the context of the measure space \( (X,\mathcal{P}(X),\mu) \). This allows us to conclude that \( \Phi \in C^1((0,+\infty)) \), with
\begin{equation}\label{Phi_dot_G}
\Phi'(t) = \frac{d}{d t} \sum_{x \in X} \varphi(x) \, u(x,t) \, \mu(x) = \sum_{x \in X} \varphi(x) \, u_t(x,t) \, \mu(x) \hspace{2.0em} \text{for all } t > 0,
\end{equation}
\noindent where the first equality exploits \eqref{def_Phi_G}.

\medskip
\noindent Now, since $f$ is convex, \( \varphi > 0 \) in \( X \) and \( \|\varphi\|_1 = 1 \), we can apply Jensen's Inequality, in the context of the measure space \( (X,\mathcal{P}(X), \mu) \), to the last sum in \eqref{second_G}, obtaining for all \( t > 0 \):
\begin{equation}\label{JI_G}
\begin{aligned}
\sum_{x \in X} \varphi(x) \, f(u(x,t)) \, \mu(x) \geq f \left( \sum_{x \in X} \varphi(x) \, u(x,t) \, \mu(x) \right)
= f(\Phi(t)).
\end{aligned}
\end{equation}

\noindent Combining the results obtained in \eqref{Phi_dot_G} and \eqref{JI_G}, and using the definition \eqref{def_Phi_G}, \eqref{second_G} implies:
\[
\Phi'(t) + \lambda \Phi(t) \geq f(\Phi(t)) \hspace{1.6em} \text{for all } t > 0,
\]
\noindent which is a first order differential inequality associated with the initial condition \eqref{def_Phi_0_G}, namely
\[
\Phi(0) > s_0(\lambda).
\]
\noindent Lemma \ref{lemmaode} implies that \( \Phi \) develops a vertical asymptote; that is, there exists some \( T_{\Phi} \in (0,+\infty) \) such that
\[
\lim_{t \to T_{\Phi}^-} \Phi(t) = +\infty.
\]

\noindent This fact is clearly in contradiction with \eqref{contradiction_1_G}. Therefore, \( u \) cannot be a global solution of problem \eqref{problem_HR_G}, as initially assumed. In conclusion, \( u \) blows up in finite time. \hfill \( \square \)

\begin{remark}\label{remark_proposition_blow_up_chapter_3}
\noindent Considering the same framework and notation in the statement of \emph{\cref{theorem_blow_up_chapter_3}}, we now show that, if the weighted graph \( (X,\omega,\mu) \) satisfies a specific condition, then hypothesis \eqref{assumption_proposition_blow_up_chapter_3} is automatically verified. In particular, after fixing two arbitrary time instants \( t_0,T' \) such that \( 0 < t_0 < T' < T \), let us set
\[
M := \| u \|_{L^\infty(X \times (t_0,T'))} = \sup_{(x,t) \in X \times (t_0,T')} u(x,t),
\]
\noindent where in the last equality we have exploited the fact that \( u \) is nonnegative. We now make use of condition \eqref{reg_sol_2_G} in \emph{\cref{definition_sub_super_sols_G}}, in order to obtain the following:
\begin{equation}\label{estimate_M_remark}
0 \leq M \leq \sup_{(x,t) \in X \times [0,T']} u(x,t) = \| u \|_{L^\infty(X \times [0,T'])} < +\infty.
\end{equation}
\noindent Now, assume that the weighted graph \( (X,\omega,\mu) \) has bounded weighted degree, namely:
\begin{equation}\label{assumptions_graph_hypothesis_verified}
D := \sup_{x \in X} \operatorname{Deg}(x) = \sup_{x \in X} \left[ \frac{1}{\mu(x)} \sum_{y \in X} \omega(x,y) \right] < +\infty.
\end{equation}
\noindent Then, for all \( (x,t) \in X \times (t_0,T') \) it holds:
\[
\begin{aligned}
\left| \Delta u(x,t) \right| & = \frac{1}{\mu(x)} \left| \sum_{y \in X} \left[ u(y,t) - u(x,t) \right] \omega(x,y) \right| \\
& \leq \frac{1}{\mu(x)} \sum_{y \in X} \left| u(y,t) - u(x,t) \right| \omega(x,y) \\
& \leq \frac{1}{\mu(x)} \sum_{y \in X} \underbrace{\left[ u(y,t) + u(x,t) \right]}_{\leq 2M} \omega(x,y) \\
& \leq \frac{2 M}{\mu(x)} \sum_{y \in X} \omega(x,y) \\
& \equiv 2 M \, \operatorname{Deg}(x) \\
& \leq 2 M D.
\end{aligned}
\]
\noindent These passages are justified by the positivity of \( \mu \), the nonnegativity of both \( u \) and \( \omega \), and the definition of both \( M \) and \( D \).

\medskip
\noindent Now, since \( u \) is a solution to problem \eqref{problem_HR_G}, it follows that, for all \( (x,t) \in X \times (t_0,T') \):
\[
\left| u_t(x,t) \right| = \left| \Delta u(x,t) + f(u(x,t)) \right| \leq \left| \Delta u(x,t) \right| + f(u(x,t)) \leq 2 M D + f(M),
\]
where we have exploited the nonnegativity of the function \( f \), the definition of \( M \) and the fact that, as already remarked in the proof of \emph{\cref{theorem_blow_up_chapter_3}}, \( f \) is increasing on \( [0,+\infty) \). Therefore, we obtain the following estimate:
\[
\sup_{t \in (t_0,T')} \left| u_t(x,t) \right| \leq C \hspace{2.0em} \text{for all } x \in X,
\]
\noindent where we have set
\[
C := 2 M D + f(M).
\]
\noindent We now notice that \( C \in [0,+\infty) \), thanks to both \eqref{estimate_M_remark} and \eqref{assumptions_graph_hypothesis_verified}. Moreover, from the last estimate it follows:
\[
\sum_{x \in X} \varphi(x) \, \sup_{t \in (t_0,T')} \left| u_t(x,t) \right| \, \mu(x) \leq C \sum_{x \in X} \varphi(x) \, \mu(x) \equiv C \, \| \varphi \|_1 = C < +\infty,
\]
\noindent where the last equality is due to the fact that \( \varphi \) satisfies property \emph{(\(\mathcal{B}_G\))}. We then infer that
\[
\varphi \, \sup_{t \in (t_0,T')} \left| u_t(\cdot,t) \right| \in \ell^1(X,\mu),
\]
\noindent and the arbitrariness of the time instants \( t_0 \) and \( T' \) yields that condition \eqref{assumption_proposition_blow_up_chapter_3} is satisfied.

\medskip
\noindent In conclusion, if the weighted graph \( (X,\omega,\mu) \) satisfies the hypothesis \eqref{assumptions_graph_hypothesis_verified}, then the assumption \eqref{assumption_proposition_blow_up_chapter_3} in the statement of \emph{\cref{theorem_blow_up_chapter_3}} is automatically verified.
\end{remark}


\section{Proof of \cref{theorem_blow_up_general}}\label{subsection_proofs_2}

\noindent In order to show \cref{theorem_blow_up_general}, we apply \cref{theorem_blow_up_chapter_3}. To do this, we are left with constructing a function satisfying property (\(\mathcal{B}_G\)), according to \cref{def_TF_G}.

\begin{lemma}\label{lemma_ell_1}
\noindent Let \( (X,\omega,\mu) \) satisfy \emph{\cref{assumption_graph_pseudo_metric}} and be endowed with the combinatorial distance \( \rho \). Let \( \Omega \subset X \) be a nonempty finite subset, and assume that \( \psi \in C(X) \) is a spherically symmetric function with respect to \( \Omega \). Consider the function \( S \) defined in \eqref{function_S}.

\medskip
\noindent If the series
\[
\sum_{r = 0}^{+\infty} |\psi(r)| S(r)
\]
\noindent is convergent, then \( \psi \in \ell^1(X,\mu) \).
\end{lemma}

\noindent \emph{Proof.} We first notice that the function \( S \) actually maps \( \mathbb{N}_0 \) into \( (0,+\infty) \). Indeed, after fixing an arbitrary \( r \in \mathbb{N}_0 \), we can apply result b) of \cref{lemma_partition_implications_rho}, since \( (X,\omega,\mu) \) is assumed to satisfy \cref{assumption_graph_pseudo_metric}, hence it is connected and also locally finite, by \cref{remark_assumptions_G}. Now, result b) in \cref{lemma_partition_implications_rho} ensures that the shell \( S_r(\Omega) \) is nonempty and finite. In particular, since the node measure \( \mu \) maps \( X \) into \( (0,+\infty) \), we conclude that \( S(r) \in (0,+\infty) \).

\medskip
\noindent Now, we recall from \cref{remark_partition_X_shells} that the family \( \{ S_r(\Omega) \}_{r=0}^{+\infty} \) is a partition of the vertex set \( X \). This allows us to perform the following computations:
\[
\begin{aligned}
\| \psi \|_1 & = \sum_{x \in X} \left| \psi(x) \right| \hspace{0.1em} \mu(x) \\
& = \sum_{r=0}^{+\infty} \sum_{x \in S_r(\Omega)} \left| \psi(x) \right| \hspace{0.1em} \mu(x) \\
& = \sum_{r=0}^{+\infty} \left| \psi(r) \right| \left[ \sum_{x \in S_r(\Omega)} \mu(x) \right] \\
& = \sum_{r=0}^{+\infty} \left| \psi(r) \right| S(r),
\end{aligned}
\]
\noindent where the third equality comes from the spherical symmetry of \( \psi \) with respect to \( \Omega \), while the last equality exploits \eqref{function_S}. Now, since the series at the right-hand side is assumed to be finite, then it follows \( \| \psi \|_1 < +\infty \), meaning that \( \psi \in \ell^1(X,\mu) \). \hfill \( \square \)

\begin{lemma}\label{lemma_double_sum}
\noindent Let the assumptions of \emph{\cref{lemma_ell_1}} be fulfilled, and moreover assume that
\[
\sup_{x \in X} \hspace{0.1em} \mathfrak{D}_+(x) < +\infty \hspace{1.8em} \text{and} \hspace{1.8em} \sup_{x \in X} \hspace{0.1em} \mathfrak{D}_-(x) < +\infty.
\]
\noindent If the series
\[
\sum_{r=0}^{+\infty} S(r) \left[ | \psi(r+1) - \psi(r) | + | \psi(r-1) - \psi(r) | \right]
\]
\noindent is convergent, then it holds:
\[
\sum_{x,y \in X} \omega(x,y) \hspace{0.1em} | \nabla_{xy} \psi | < +\infty.
\]
\end{lemma}

\noindent \emph{Proof.} First, let us consider an arbitrary function \( h \in C(X) \) being spherically symmetric with respect to \( \Omega \). By Lemma \ref{lemma_formula_Laplacian_for_WSS_graphs} we obtain the following formula, valid for all \( x \in X \):
\[
\begin{aligned}
\mu(x) \hspace{0.05em} \Delta h(x) & = \sum_{y \in X} \omega(x,y) \left[ \nabla_{xy} h \right] \\
& = \mu(x) \hspace{0.05em} \left\{ \mathfrak{D}_+(x) \left[ h(r(x)+1) - h(r(x)) \right] + \mathfrak{D}_-(x) \left[ h(r(x)-1) - h(r(x)) \right] \right\}.
\end{aligned}
\]
\noindent In particular, by inserting a modulus inside the sum and repeating the exact same passages as those present in the proof of \cref{lemma_formula_Laplacian_for_WSS_graphs}, the following identity is obtained:
\begin{equation}\label{identity_Laplacian_f_modulus}
\begin{aligned}
\sum_{y \in X} \omega(x,y) \left| \nabla_{xy} h \right| & = \mu(x) \, \mathfrak{D}_+(x) \left| h(r(x)+1) - h(r(x)) \right| \hspace{0.4em} + \\
& \hspace{1.1em} + \mu(x) \, \mathfrak{D}_-(x) \left| h(r(x)-1) - h(r(x)) \right|,
\end{aligned}
\end{equation}
\noindent holding for all \( x \in X \). Notice that this identity can be used with the function \( \psi \), which is assumed to be spherically symmetric with respect to \( \Omega \). Therefore, after setting
\[
D_+ := \sup_{x \in X} \hspace{0.1em} \mathfrak{D}_+(x)
\]
\noindent and
\[
D_- := \sup_{x \in X} \hspace{0.1em} \mathfrak{D}_-(x),
\]
\noindent the following computations can be performed:
\begin{equation}\label{passages_lemma_double_sum}
\begin{aligned}
\sum_{x,y \in X} \omega(x,y) \hspace{0.1em} \left| \nabla_{xy} \psi \right| & = \sum_{x \in X} \left[ \sum_{y \in X} \omega(x,y) \hspace{0.1em} \left| \nabla_{xy} \psi \right| \right] \\
& = \sum_{x \in X} \mu(x) \, \mathfrak{D}_+(x) \left| \psi(r(x)+1) - \psi(r(x)) \right| \hspace{0.4em} + \\
& \hspace{1.1em} + \sum_{x \in X} \mu(x) \, \mathfrak{D}_-(x) \left| \psi(r(x)-1) - \psi(r(x)) \right| \\
& \leq D_+ \sum_{x \in X} \mu(x) \hspace{0.05em} \left| \psi(r(x)+1) - \psi(r(x)) \right| \hspace{0.4em} + \\
& \hspace{1.1em} + D_- \sum_{x \in X} \mu(x) \hspace{0.05em} \left| \psi(r(x)-1) - \psi(r(x)) \right| \\
& \leq \max \left\{ D_+, D_- \right\} \sum_{x \in X} \mu(x) \, \left| \psi(r(x)+1) - \psi(r(x)) \right| \hspace{0.4em} + \\
& \hspace{1.1em} + \max \left\{ D_+, D_- \right\} \sum_{x \in X} \mu(x) \, \left| \psi(r(x)-1) - \psi(r(x)) \right| \\
& = C \sum_{r=0}^{+\infty} \sum_{x \in S_r(\Omega)} \mu(x) \hspace{0.05em} \left\{ \left| \psi(r+1) - \psi(r) \right| + \left| \psi(r-1) - \psi(r) \right| \right\} \\
& = C \sum_{r=0}^{+\infty} \left\{ \left| \psi(r+1) - \psi(r) \right| + \left| \psi(r-1) - \psi(r) \right| \right\} \left[ \sum_{x \in S_r(\Omega)} \mu(x) \right] \\
& \equiv C \sum_{r=0}^{+\infty} S(r) \hspace{0.05em} \left\{ \left| \psi(r+1) - \psi(r) \right| + \left| \psi(r-1) - \psi(r) \right| \right\},
\end{aligned}
\end{equation}
\noindent where we have set
\[
C := \max \left\{ D_+, D_- \right\} \equiv \max \left\{ \sup_{x \in X} \hspace{0.1em} \mathfrak{D}_+(x), \sup_{x \in X} \hspace{0.1em} \mathfrak{D}_-(x) \right\},
\]
\noindent which belongs to \( [0,+\infty) \), thanks to the assumption. In \eqref{passages_lemma_double_sum}, the second equality is justified by \eqref{identity_Laplacian_f_modulus}, while the third one exploits both the spherical symmetry of \( \psi \) and the fact that, as stated in \cref{remark_partition_X_shells}, the family \( \{ S_r(\Omega) \}_{r=0}^{+\infty} \) is a partition of the vertex set \( X \). The other passages in \eqref{passages_lemma_double_sum} are immediate.

\medskip
\noindent Finally, thanks to the fact that \( C \in [0,+\infty) \), and making use of the finiteness of the series in the assumption, it follows that the term at the right-hand side of \eqref{passages_lemma_double_sum} is finite, which proves the thesis. \hfill \( \square \)

\begin{lemma}\label{lemma_laplacian_lambda}
\noindent Let the assumptions of \emph{\cref{lemma_ell_1}} be fulfilled, and define the function \( \psi \) as
\begin{equation}\label{def_psi_radial}
\psi(x) := e^{- a \hspace{0.1em} r(x)} \hspace{1.8em} \text{for all } \hspace{0.1em} x \text{ in } X,
\end{equation}
\noindent where \( a > 0 \) is a parameter. Furthermore, assume that
\[
\sup_{x \in X} \hspace{0.1em} \mathfrak{D}_+(x) < +\infty.
\]
\noindent Then the following inequality holds:
\[
\Delta \psi(x) + \lambda \psi(x) \geq 0 \hspace{1.8em} \text{for all } \hspace{0.1em} x \text{ in } X,
\]
\noindent provided that
\[
\lambda \geq \sup_{x \in X} \hspace{0.1em} \mathfrak{D}_+(x).
\]
\end{lemma}

\noindent \emph{Proof.} We first notice that the analytical expression of the function \( \psi \) depends on the variable only through \( r(\cdot) = \rho(\cdot,\Omega) \), and therefore it is immediate to verify that \( \psi \) is spherically symmetric with respect to \( \Omega \).

\medskip
\noindent Moreover, as already specified at the beginning of the proof of the previous lemma, by Lemma \ref{lemma_formula_Laplacian_for_WSS_graphs}  we obtain the following formula for the function \( \psi \in C(X) \), which is spherically symmetric with respect to \( \Omega \):
\begin{equation}\label{identity_Laplacian_psi}
\Delta \psi(x) = \mathfrak{D}_+(x) \left[ \psi(r(x)+1) - \psi(r(x)) \right] + \mathfrak{D}_-(x) \left[ \psi(r(x)-1) - \psi(r(x)) \right],
\end{equation}
\noindent holding for all \( x \in X \). Then we can write, for all \( x \in X \):
\[
\begin{aligned}
\Delta \psi(x) & = \mathfrak{D}_+(x) \left[ e^{- a \hspace{0.1em} (r(x)+1)} - e^{- a \hspace{0.1em} r(x)} \right] + \mathfrak{D}_-(x) \left[ e^{- a \hspace{0.1em} (r(x)-1)} - e^{- a \hspace{0.1em} r(x)} \right] \\
& \geq \mathfrak{D}_+(x) \left[ e^{- a \hspace{0.1em} (r(x)+1)} - e^{- a \hspace{0.1em} r(x)} \right].
\end{aligned}
\]
\noindent Here, the equality exploits \eqref{identity_Laplacian_psi}, together with the definition of \( \psi \) and the characteristic abuse of notation introduced when dealing with spherically symmetric functions. On the other hand, the inequality is a consequence of the fact that the function \( e^{-a (\cdot)} : [0,+\infty) \to (0,+\infty) \) is decreasing, and that the inner degree is nonnegative.

\medskip
\noindent Now, consider \( \lambda \) as in the statement, so that
\[
\lambda \geq \sup_{x \in X} \hspace{0.1em} \mathfrak{D}_+(x) \geq \mathfrak{D}_+(x), \hspace{1.8em} \text{for all } x \in X.
\]
\noindent From the last computations it directly follows that, for all \( x \in X \):
\[
\begin{aligned}
\Delta \psi(x) + \lambda \psi(x) & \geq \mathfrak{D}_+(x) \left[ e^{- a \hspace{0.1em} (r(x)+1)} - e^{- a \hspace{0.1em} r(x)} \right] + \lambda \hspace{0.1em} e^{- a \hspace{0.1em} r(x)} \\
& \equiv \underbrace{\mathfrak{D}_+(x)}_{\geq 0} \hspace{0.1em} \underbrace{e^{- a \hspace{0.1em} (r(x)+1)}}_{>0} + \underbrace{e^{- a \hspace{0.1em} r(x)}}_{>0} \underbrace{\left[ \lambda - \mathfrak{D}_+(x) \right]}_{\geq 0} \\
& \geq 0.
\end{aligned}
\]
\noindent This ends the proof. \hfill \( \square \)

\begin{corollary}\label{corollary_property_B_G_general_graphs_1}
\noindent Let \( (X,\omega,\mu) \) be a weighted graph satisfying \emph{\cref{assumption_graph_pseudo_metric}}, endowed with the combinatorial distance \( \rho \). Let \( \Omega \subset X \) be a nonempty finite subset and define the function \( S \) as in \eqref{function_S}. Furthermore, assume that conditions (a)-(b) in \emph{\cref{theorem_blow_up_general}} are fulfilled.

\noindent Then the function \( \varphi \) defined in \eqref{definition_varphi_barrier_chapter 3} satisfies the property \emph{(\(\mathcal{B}_G\))}, provided that
\[
\lambda \geq \sup_{x \in X} \hspace{0.1em} \mathfrak{D}_+(x).
\]
\end{corollary}

\noindent \emph{Proof.} We first notice that the constant \( C \) appearing in \eqref{definition_varphi_barrier_chapter 3} belongs to \( (0,+\infty) \), thanks to the strict positivity of the function \( S \) and to assumption b) in \cref{theorem_blow_up_general}. This allows us to verify that \( \varphi \) maps \( X \) into \( (0,+\infty) \), which corresponds to condition a) of property (\(\mathcal{B}_G\)), according to \cref{def_TF_G}.

\medskip
\noindent Moreover, the following relation holds between \( \varphi \) and the function \( \psi \) defined in \eqref{def_psi_radial}:
\begin{equation}\label{relation_varphi_psi}
\varphi(x) = C \hspace{0.1em} \psi(x), \hspace{1.0em} \text{for all } x \in X.
\end{equation}
\noindent Now, in the proof of \cref{lemma_laplacian_lambda} the function \( \psi \) has already been proved to be spherically symmetric with respect to \( \Omega \). Therefore, by exploiting \eqref{relation_varphi_psi} it is immediate to deduce that also \( \varphi \) is spherically symmetric. In addition, it holds:
\begin{equation}\label{varphi_sum_equal_to_1}
\sum_{r = 0}^{+\infty} |\varphi(r)| S(r) \equiv \sum_{r = 0}^{+\infty} \varphi(r) \hspace{0.1em} S(r) = \sum_{r = 0}^{+\infty} C \hspace{0.05em} e^{-a \hspace{0.05em} r} S(r) \equiv C \left[ \sum_{r = 0}^{+\infty} e^{-a \hspace{0.05em} r} S(r) \right] = \frac{C}{C} \equiv 1.
\end{equation}
\noindent Therefore, all the assumptions of \cref{lemma_ell_1} are satisfied, yielding that \( \varphi \in \ell^1(X,\mu) \). Furthermore, by arguing in the same way as in the proof of \cref{lemma_ell_1}, we get:
\[
\| \varphi \|_1 = \sum_{r=0}^{+\infty} \left| \varphi(r) \right| S(r) = 1,
\]
\noindent where the last equality comes from \eqref{varphi_sum_equal_to_1}. In particular, condition b) of property (\(\mathcal{B}_G\)) is satisfied by \( \varphi \).

\medskip
\noindent Now, for all \( r \in \mathbb{N}_0 \) we can write:
\[
\begin{aligned}
| \psi(r+1) - \psi(r) | + | \psi(r-1) - \psi(r) | & = | e^{-a \hspace{0.05em} (r+1)} - e^{-a \hspace{0.05em} r} | + | e^{-a \hspace{0.05em} (r-1)} - e^{-a \hspace{0.05em} r} | \\
& = e^{-a \hspace{0.05em} r} | e^{-a} - 1 | + e^{-a \hspace{0.05em} r} | e^a - 1 | \\
& = e^{-a \hspace{0.05em} r} \left[ 1 - e^{-a} \right] + e^{-a \hspace{0.05em} r} \left[ e^a - 1 \right] \\
& = e^{-a \hspace{0.05em} r} \left[ e^a - e^{-a} \right],
\end{aligned}
\]
\noindent where we have used the definition of the function \( \psi \) and the fact that \( a > 0 \). By summing the obtained equality over all \( r \in \mathbb{N}_0 \), we infer:
\[
\sum_{r=0}^{+\infty} S(r) \left[ | \psi(r+1) - \psi(r) | + | \psi(r-1) - \psi(r) | \right] = \left( e^a - e^{-a} \right) \sum_{r=0}^{+\infty} S(r) e^{-a \hspace{0.05em} r} < +\infty,
\]
\noindent where the inequality comes from assumption b) in \cref{theorem_blow_up_general}. Therefore, after multiplying by \( C \in (0,+\infty) \), we can exploit the relation \eqref{relation_varphi_psi} in order to obtain:
\[
\sum_{r=0}^{+\infty} S(r) \left[ | \varphi(r+1) - \varphi(r) | + | \varphi(r-1) - \varphi(r) | \right] < +\infty.
\]
\noindent This fact, together with assumption a) in \cref{theorem_blow_up_general}, allows us to apply \cref{lemma_double_sum}, yielding:
\[
\sum_{x,y \in X} \omega(x,y) \hspace{0.1em} | \nabla_{xy} \varphi | < +\infty.
\]
\noindent In particular, after fixing two arbitrary values \( R > 0 \) and \( \delta \in (0,1) \), this sum is independent of both \( R \) and \( \delta \), and the following holds:
\[
\sum_{\substack{x,y \in X \\ \hspace{1.0em} x \in \overline{B}_R \setminus B_{(1-\delta)R - 2s}}} \omega(x,y) \hspace{0.2em} |\nabla_{xy} \varphi| \hspace{0.1em}
\leq \hspace{0.1em} \sum_{x,y \in X} \omega(x,y) \hspace{0.1em} | \nabla_{xy} \varphi | < +\infty,
\]
\noindent which implies that also condition c) of property (\(\mathcal{B}_G\)) is satisfied by \( \varphi \).

\medskip
\noindent Finally, we notice that, under our hypotheses, \cref{lemma_laplacian_lambda} can be applied, yielding:
\[
\Delta \psi(x) + \lambda \psi(x) \geq 0 \hspace{1.8em} \text{for all } \hspace{0.1em} x \text{ in } X.
\]
\noindent By exploiting both the linearity of the Laplacian operator and the relation \eqref{relation_varphi_psi}, we finally obtain, for all \( x \in X \):
\[
\Delta \varphi(x) + \lambda \varphi(x) = C \left[ \Delta \psi(x) + \lambda \psi(x) \right] \geq 0,
\]
\noindent where the inequality is due to the fact that \( C \in (0,+\infty) \). In conclusion, the function \( \varphi \) satisfies condition d) of property (\(\mathcal{B}_G\)), according to \cref{def_TF_G}. This finishes the proof. \hfill \( \square \)

\medskip
\noindent \emph{Proof of \emph{\cref{theorem_blow_up_general}}}. The proof follows directly from \cref{theorem_blow_up_chapter_3} and \cref{corollary_property_B_G_general_graphs_1}. \hfill \( \square \)

\begin{remark}\label{remark_general_graphs_blow_up}
\noindent By employing \emph{\cref{remark_proposition_blow_up_chapter_3}}, we observe that, if in addition the weighted graph \( (X,\omega,\mu) \) satisfies the assumption \eqref{assumptions_graph_hypothesis_verified}, namely it has bounded weighted degree, then the integrability condition \eqref{assumption_proposition_blow_up_chapter_3} in the statement of \emph{\cref{theorem_blow_up_general}} is automatically satisfied.
\end{remark}


\section{Proofs of \cref{theorem_blow_up_model_trees} and \cref{theorem_blow_up_homogeneous_trees}}\label{subsection_proofs_3}

\noindent We now apply Kaplan's method in the context of model trees. We first need to exhibit a function satisfying the property (\(\mathcal{B}_G\)), according to \cref{def_TF_G}.

\begin{lemma}\label{corollary_property_B_G_model_trees}
\noindent Let \( (\mathbb{T},\omega_0,\mu_1) \) be a model tree, with root \( x_0 \in \mathbb{T} \) and branching function \( b \). Suppose that \eqref{assumption_sup_branching_finite} holds, namely:
\[
B = \sup_{r \in \mathbb{N}_0} b(r) < +\infty.
\]
\noindent Then, for any constant \( a > \log(B) \), the function \( \varphi \) defined in \eqref{definition_varphi_model_trees} satisfies the property \emph{(\(\mathcal{B}_G\))}, provided that \( \lambda \geq B \).
\end{lemma}

\noindent \emph{Proof.} We aim to apply \cref{corollary_property_B_G_general_graphs_1} to the current framework. In order to do so, we first remark that, in the specific context of model trees, we set \( \Omega = \{x_0\} \), corresponding to a nonempty, finite subset of \( \mathbb{T} \). Furthermore, model trees are naturally endowed with the combinatorial distance \( \rho \). We then need to verify that, under our hypotheses, the model tree \( (\mathbb{T},\omega_0,\mu_1) \) satisfies \cref{assumption_graph_pseudo_metric}, and also that both conditions a) and b) in \cref{theorem_blow_up_general} are satisfied.

\medskip
\noindent First, we have that \( (\mathbb{T},\omega_0,\mu_1) \) is connected and locally finite. Furthermore, we recall that, by the definition of the combinatorial distance, for every couple of neighboring nodes \( x,y \in \mathbb{T} \) it holds \( \rho(x,y) = 1 \). Hence the jump size satisfies
\[
s = \sup\{ \rho(x,y) : x, y \in \mathbb{T}, \text{ } x \sim y \} = 1.
\]
\noindent In particular, \( s \in (0,+\infty) \).

\medskip
\noindent Since \( (\mathbb{T},\omega_0,\mu_1) \) is locally finite, result b) in \cref{lemma_partition_implications_rho} can be applied. Specifically, after fixing an arbitrary node \( x^* \in \mathbb{T} \), then for each \( r \in \mathbb{N}_0 \), the shell of radius \( r \) centered at \( x^* \), namely
\[
S_r(\{x^*\}) = \left\{ x \in \mathbb{T} : \rho(x,\{x^*\}) = r \right\},
\]
\noindent is nonempty and finite. Now, notice that, given an arbitrary \( r > 0 \), the ball of radius \( r \) centred at \( \{x_0\} \) can be rewritten as a finite union of such shells, namely:
\[
\begin{aligned}
B_r(x^*) & = \{ x \in \mathbb{T} : \rho(x, x^*) < r \} \\
& = \{ x \in \mathbb{T} : \rho(x, x^*) \leq \lceil r \rceil - 1 \} \\
& = \bigcup_{k=0}^{\lceil r \rceil - 1} \{ x \in \mathbb{T} : \rho(x, x^*) = k \} \\
& = \bigcup_{k=0}^{\lceil r \rceil - 1} S_k(\{x^*\}),
\end{aligned}
\]
\noindent where \( \lceil \cdot \rceil \) represent the ceiling function. Here, the second and the third equalities hold since the combinatorial distance has codomain \( \mathbb{N}_0 \). Now, thanks to the fact that each shell is nonempty and finite, then we obtain the finiteness of \( B_r(x^*) \) for any \( r > 0 \).

\medskip
\noindent Now, for any vertex \( x \in \mathbb{T} \) we compute:
\begin{equation}\label{passages_degree_model_tree_chapter_3}
\frac{1}{\mu_1(x)} \sum_{y \in \mathbb{T}} \omega_0(x,y) \hspace{0.05em} \rho(x,y) = \sum_{\substack{y \in \mathbb{T} \\ y \sim x}} \omega_0(x,y) \hspace{0.05em} \rho(x,y) = \left| \left\{ y \in \mathbb{T} : y \sim x \right\} \right| = \deg(x).
\end{equation}
\noindent Here, the first equality holds since \( \mu_1 \equiv 1 \) in \( \mathbb{T} \), and because \( \omega_0 \) is null on nonadjacent vertices. The second equality, instead, is due to the fact that both \( \omega_0 \) and \( \rho \) assume the value \( 1 \) when evaluated on neighboring vertices.

\medskip
\noindent Applying \cref{remark_partition_X_shells} to the current context yields that the family \( \{ S_r(x_0) \}_{r=0}^{+\infty} \) is a partition of the vertex set \( \mathbb{T} \). Together with the assumption \eqref{assumption_sup_branching_finite}, this fact implies:
\begin{equation}\label{identity_sup_b_r}
\sup_{x \in \mathbb{T}} b(r(x)) = \sup_{r \in \mathbb{N}_0} b(r) < +\infty.
\end{equation}
\noindent In particular, since
\[
\deg(x) =
\begin{cases}
b(0) & \text{ if } x = x_0 \\
b(r(x)) + 1 & \text{ if } x \in \mathbb{T} \setminus \{ x_0 \},
\end{cases}
\]
\noindent then \( \deg(x) \geq 1 \) for all \( x \in \mathbb{T} \), and it also holds:
\[
1 \leq \sup_{x \in \mathbb{T}} \hspace{0.1em} \deg(x) < +\infty.
\]
\noindent By \eqref{passages_degree_model_tree_chapter_3} we then deduce:
\[
\frac{1}{\mu_1(x)} \sum_{y \in \mathbb{T}} \omega_0(x,y) \hspace{0.05em} \rho(x,y) \leq \sup_{x \in \mathbb{T}} \deg(x) \hspace{2.5em} \text{for all } x \in \mathbb{T},
\]
\noindent with \( \sup_{x \in \mathbb{T}} \deg(x) \in [1,+\infty) \). According to \cref{definition_intrinsic}, this means that \( \rho \) is \( 1 \)-intrinsic with bound \( \sup_{x \in \mathbb{T}} \deg(x) \).

\medskip
\noindent In conclusion, under the current hypotheses, the metric \( \rho \) on \( \mathbb{T} \) satisfies property (PM), according to \cref{definition_PM}, and \cref{assumption_graph_pseudo_metric} is respected by the model tree \( (\mathbb{T},\omega_0,\mu_1) \), endowed with \( \rho \).

\medskip
\noindent Now, we recall that in this context the inner degree \( \mathfrak{D}_- \) with respect to the set \( \{x_0\} \) has codomain \( \{0,1\} \), hence
\[
\sup_{x \in \mathbb{T}} \hspace{0.1em} \mathfrak{D}_-(x) = 1 < +\infty.
\]
\noindent Furthermore, the relation \( \mathfrak{D}_+(\cdot) = b(r(\cdot)) \), holding in the whole vertex set \( \mathbb{T} \), yields the following:
\begin{equation}\label{finiteness_sup_D_+_model_trees}
\sup_{x \in \mathbb{T}} \hspace{0.1em} \mathfrak{D}_+(x) = \sup_{x \in \mathbb{T}} b(r(x)) < +\infty,
\end{equation}
\noindent where the last inequality comes from \eqref{identity_sup_b_r}. Therefore, condition a) in \cref{theorem_blow_up_general} is satisfied.

\medskip
\noindent We now investigate the analytic expression of the function \( S : \mathbb{N}_0 \to (0,+\infty) \) introduced in \eqref{function_S}. More specifically, since here \( \mu_1 \equiv 1 \) and \( \Omega = \{ x_0 \} \), for any \( r \in \mathbb{N}_0 \) we obtain:
\begin{equation}\label{expression_S_r_model_trees}
S(r) = \sum_{x \in S_r(x_0)} \mu_1(x) \equiv \sum_{x \in S_r(x_0)} 1 = |S_r(x_0)|.
\end{equation}
\noindent Since \( S_0(x_0) = \{x_0\} \), it holds \( S(0) = |\{x_0\}| = 1 \), while exploiting \eqref{property_0_1_second_way_model_trees} we infer that
\[
S(1) = |S_1(x_0)| = \left| \left\{ y \in \mathbb{T} : y \sim x_0 \right\} \right| = \deg(x_0) = b(0).
\]
\noindent Let us fix for a moment a radius \( r \in \mathbb{N}_0 \). We know that each vertex in \( S_{r+1}(x_0) \) has only one neighbor belonging to the preceding shell \( S_r(x_0) \). Hence there are no two vertices in \( S_r(x_0) \) sharing the same neighbor in \( S_{r+1}(x_0) \). This means that, in order to count the number of vertices belonging to \( S_{r+1}(x_0) \), we simply need to multiply the number of vertices contained in \( S_r(x_0) \) by the number of neighbors that each one of such vertices has in the shell \( S_{r+1}(x_0) \), namely
\[
\mathfrak{D}_+(r) = b(r).
\]
\noindent We deduce the following recurrence relation:
\[
S(r+1) = \left| S_{r+1}(x_0) \right| = b(r) \left| S_r(x_0) \right| = b(r) \hspace{0.05em} S(r),
\]
\noindent where the first and the last equalities are justified by \eqref{expression_S_r_model_trees}. By iterating this identity we get, for all \( r \in \mathbb{N}_0 \):
\[
S(r+1) = S(r) \hspace{0.05em} b(r) = S(r-1) \hspace{0.05em} b(r-1) \hspace{0.05em} b(r) = \cdots = S(1) \prod_{k=1}^{r} b(k) = \prod_{k=0}^{r} b(k),
\]
\noindent where the last equality exploits the identity \( S(1) = b(0) \), established above. In conclusion, the analytical expression for the function \( S \) is
\begin{equation}\label{expression_S_r_model_trees_1}
\begin{aligned}
S(0) & = 1 \\
S(r) & = \prod_{k=0}^{r-1} b(k) \hspace{2.5em} \text{for all } r \in \mathbb{N}.
\end{aligned}
\end{equation}
\noindent We now notice that, by the definition of the supremum and thanks to the assumption \eqref{assumption_sup_branching_finite}, for all \( r \in \mathbb{N}_0 \) it holds:
\[
1 \leq b(r) \leq \sup_{r \in \mathbb{N}_0} b(r) = B < +\infty,
\]
\noindent so that, exploiting the assumption on the constant \( a \), we obtain \( a \in (0,+\infty) \). In addition, we deduce the following estimate:
\[
1 \leq \prod_{k=0}^{r-1} b(k) \leq B^r,
\]
\noindent holding for all \( r \in \mathbb{N} \). Thanks to \eqref{expression_S_r_model_trees_1}, this fact allows us to derive:
\[
1 \leq S(r) \leq B^r, \hspace{2.5em} \text{for all } r \in \mathbb{N}_0,
\]
\noindent which simply corresponds to an identity in the case \( r = 0 \). By exploiting this result, we conclude:
\[
0 < \sum_{r=0}^{+\infty} S(r) \hspace{0.05em} e^{-a \hspace{0.05em} r} \leq \sum_{r=0}^{+\infty} B^r \hspace{0.05em} e^{-a \hspace{0.05em} r} \equiv
\sum_{r=0}^{+\infty} \left[ B \hspace{0.05em} e^{-a} \right]^r.
\]
\noindent We notice that the term at the right-hand side corresponds to a geometric series with common ratio \( B \hspace{0.05em} e^{-a} \). Since by assumption \( a > \log(B) \), it holds:
\[
0 < B \hspace{0.05em} e^{-a} < B \cdot \frac{1}{B} = 1,
\]
\noindent and therefore, by classical arguments, the geometric series converges. By the previous estimates, it follows:
\[
0 < \sum_{r=0}^{+\infty} S(r) \hspace{0.05em} e^{-a \hspace{0.05em} r} < +\infty,
\]
\noindent meaning that also condition b) in \cref{theorem_blow_up_general} is satisfied in the current framework.

\medskip
\noindent In addition, we notice that our assumption on the parameter \( \lambda \) reads as follows:
\[
\lambda \geq B = \sup_{r \in \mathbb{N}_0} b(r) = \sup_{x \in \mathbb{T}} \mathfrak{D}_+(x),
\]
\noindent where the last equality combines \eqref{identity_sup_b_r} and \eqref{finiteness_sup_D_+_model_trees}. Thus also the condition imposed on \( \lambda \) in \cref{corollary_property_B_G_general_graphs_1} is here respected. Therefore, by applying \cref{corollary_property_B_G_general_graphs_1}, we conclude that the function \( \varphi \) defined in \eqref{definition_varphi_barrier_chapter 3} satisfies the property (\(\mathcal{B}_G\)). It is then trivial to verify that this function coincides with the one introduced in the current context, namely in \eqref{definition_varphi_model_trees}. Indeed, by \eqref{expression_S_r_model_trees_1} it follows:
\[
\sum_{r=0}^{+\infty} S(r) \hspace{0.05em} e^{-a \hspace{0.05em} r} = S(0) + \sum_{r=1}^{+\infty} S(r) \hspace{0.05em} e^{-a \hspace{0.05em} r} = 1 + \sum_{r = 1}^{+\infty} \left[ \prod_{k=0}^{r-1} b(k) \right] e^{-a \hspace{0.05em} r}.
\]
\noindent In conclusion, the barrier function in \eqref{definition_varphi_model_trees} satisfies the property (\(\mathcal{B}_G\)), according to \cref{def_TF_G}. This ends the proof. \hfill \( \square \)

\medskip
\noindent \emph{Proof of \emph{\cref{theorem_blow_up_model_trees}}}. This result follows by combining \cref{theorem_blow_up_chapter_3} with \cref{corollary_property_B_G_model_trees}. Indeed, as shown in the proof of \cref{corollary_property_B_G_model_trees}, under our assumptions the model tree \( (\mathbb{T},\omega_0,\mu_1) \) verifies \cref{assumption_graph_pseudo_metric}, and in addition the function \( \varphi \) defined in \eqref{definition_varphi_model_trees} satisfies property (\(\mathcal{B}_G\)).

\medskip
\noindent Furthermore, in the proof of \cref{corollary_property_B_G_model_trees} we showed that \eqref{assumption_sup_branching_finite} implies the following estimate on the degree function:
\[
1 \leq \sup_{x \in \mathbb{T}} \hspace{0.1em} \, \deg(x) < +\infty.
\]
\noindent In particular, from \cref{definition_weighted_degree} and from the fact that \( \mu_1 \equiv 1 \) on \( \mathbb{T} \), it follows that
\[
\operatorname{Deg}(x) = \frac{\deg(x)}{\mu_1(x)} \equiv \deg(x) \hspace{2.5em} \text{for all } x \in \mathbb{T}.
\]
\noindent Therefore, we infer that
\[
\sup_{x \in \mathbb{T}} \operatorname{Deg}(x) = \sup_{x \in \mathbb{T}} \, \deg(x) < +\infty,
\]
\noindent meaning that hypothesis \eqref{assumptions_graph_hypothesis_verified} is here satisfied, and thus by \cref{remark_proposition_blow_up_chapter_3} we deduce that the integrability condition \eqref{assumption_proposition_blow_up_chapter_3} in the statement of \cref{theorem_blow_up_chapter_3} is automatically true.

\medskip
\noindent Finally, it holds:
\[
\sum_{x \in \mathbb{T}} u_0(x) \hspace{0.1em} \varphi(x) \hspace{0.1em} \mu_1(x) \equiv \sum_{x \in \mathbb{T}} u_0(x) \hspace{0.1em} \varphi(x) > s_0(\lambda),
\]
\noindent where we have, again, exploited that \( \mu_1 \equiv 1 \) on \( \mathbb{T} \). In conclusion, all the hypotheses in \cref{theorem_blow_up_chapter_3} are here satisfied, hence it follows that the solution \( u \) is nonglobal. This concludes the proof. \hfill \( \square \)

\medskip
\noindent \emph{Proof of \emph{\cref{theorem_blow_up_homogeneous_trees}}}. In the specific case of homogeneous model trees, the branching function is constant, that is, there exists \( b \in \mathbb{N} \) such that \( b(r) = b \) for all \( r \in \mathbb{N}_0 \). Therefore, it holds:
\[
\sup_{r \in \mathbb{N}_0} b(r) = b < +\infty,
\]
\noindent meaning that condition \eqref{assumption_sup_branching_finite} automatically holds.

\medskip
\noindent Moreover, it is straightforward to observe that the function introduced in \eqref{definition_varphi_homogeneous_trees} coincides with that in \eqref{definition_varphi_model_trees}, in the special case where the branching function is constant. Indeed, in this scenario it holds:
\[
\begin{aligned}
1 + \sum_{r = 1}^{+\infty} \left[ \prod_{k=0}^{r-1} b(k) \right] e^{-a \hspace{0.05em} r} & = 1 + \sum_{r = 1}^{+\infty} \left[ \prod_{k=0}^{r-1} b \right] e^{-a \hspace{0.05em} r} \\
& = 1 + \sum_{r = 1}^{+\infty} b^r \hspace{0.05em} e^{-a \hspace{0.05em} r} \\
& = 1 + \sum_{r = 1}^{+\infty} \left[ b \hspace{0.05em} e^{-a} \right]^r \\
& = \sum_{r = 0}^{+\infty} \left[ b \hspace{0.05em} e^{-a} \right]^r \\
& = \frac{1}{1 - b \hspace{0.05em} e^{-a}},
\end{aligned}
\]
\noindent where we have exploited the fact that, since by assumption \( a > \log(b) \), then
\[
0 < b \hspace{0.05em} e^{-a} < b \cdot \frac{1}{b} = 1,
\]
\noindent and therefore, by classical arguments, the obtained geometric series converges to a known value.

\medskip
\noindent These considerations allow us to apply \cref{theorem_blow_up_model_trees}, yielding the thesis. \hfill \( \square \)


\section{Proof of \cref{theorem_blow_up_lattice_1}}\label{subsection_proofs_4}

\noindent We now proceed and apply Kaplan's method in the context of the integer lattice. We do so by establishing a series of sequential results which will lead to the construction of a barrier function \( \varphi \) satisfying property (\(\mathcal{B}_G\)), according to \cref{def_TF_G}. 

\medskip
We first recall the following result (see, e.g., \cite{Elk}), regarding the theta function introduced in \eqref{theta_function}.

\begin{lemma}\label{lemma_identity_theta_function}
\noindent For each fixed \( N \in \mathbb{N} \), it holds:
\[
\sum_{x \in \mathbb{Z}^N} e^{- k \hspace{0.05em} |x|^2} =
\left[ \sum_{m \in \mathbb{Z}} e^{- k \hspace{0.05em} m^2} \right]^N =
\left[ \theta \left( \frac{k}{\pi} \right) \right]^N, \hspace{2.5em} \text{for all } \hspace{0.1em} k > 0.
\]
\end{lemma}

\begin{lemma}\label{lemma_assumption_3_1_lattice}
\noindent For each \( N \in \mathbb{N} \), the integer lattice \( (\mathbb{Z}^N,\omega_0,\mu) \) satisfies \emph{\cref{assumption_graph_pseudo_metric}}.
\end{lemma}

\noindent \emph{Proof.} Clearly, \cref{definition_PM} (a)-(b) are fulfilled. Let us show that also \cref{definition_PM} (c) holds. To this aim, for any vertex \( x \in \mathbb{Z}^N \) we compute:
\[
\sum_{y \in \mathbb{Z}^N} \omega_0(x,y) \hspace{0.05em} d(x,y) = \sum_{\substack{y \in \mathbb{Z}^N \\ y \sim x}} \omega_0(x,y) \hspace{0.05em} d(x,y) = \sum_{\substack{y \in \mathbb{Z}^N \\ y \sim x}} 1 = \left| \left\{ y \in \mathbb{Z}^N : y \sim x \right\} \right| = \deg(x) = 2 \hspace{0.05em} N.
\]
\noindent Here, the first equality holds since \( \omega_0 \) is null on nonadjacent vertices, while the second one is due to the fact that both \( \omega_0 \) and \( d \) assume the value \( 1 \) when evaluated on neighboring vertices. Since in this framework we have \( \mu \equiv 2 \hspace{0.05em} N \) in \( \mathbb{Z}^N \), then it holds:
\[
\frac{1}{\mu(x)} \sum_{y \in \mathbb{Z}^N} \omega_0(x,y) \hspace{0.05em} d(x,y) = 1, \hspace{1.8em} \text{for all } x \in \mathbb{Z}^N.
\]
\noindent According to \cref{definition_intrinsic}, this means that \( d \) is \( 1 \)-intrinsic with bound \( C_0 = 1 \). Therefore, also condition c) in \cref{definition_PM} is here satisfied.

\medskip
\noindent In conclusion, the Euclidean metric \( d \) satisfies property (PM), according to \cref{definition_PM}, and thus all the conditions in \cref{assumption_graph_pseudo_metric} are satisfied. This ends the proof. \hfill \( \square \)

\begin{lemma}\label{lemma_ell_1_lattice}
\noindent For any \( k > 0 \), the function \( \psi : \mathbb{Z}^N \to (0,+\infty) \) defined as
\begin{equation}\label{definition_psi_lattice}
\psi(x) := e^{-k \hspace{0.05em} |x|^2} \hspace{2.5em} \text{for all } \hspace{0.1em} x \text{ in } \mathbb{Z}^N
\end{equation}
\noindent belongs to \( \ell^1(\mathbb{Z}^N,\mu) \) and
\[\|\psi\|_1=2 \hspace{0.03em} N \left[ \theta \hspace{-0.25em} \left( \frac{k}{\pi} \right) \right]^N.\]
\end{lemma}

\noindent \emph{Proof.} We first notice that it is immediate to verify that the function \( \psi \) maps \( \mathbb{Z}^N \) into \( (0,+\infty) \). Now, let us fix an arbitrary \( k > 0 \). Then \cref{lemma_identity_theta_function} ensures that
\[
\sum_{x \in \mathbb{Z}^N} e^{- k \hspace{0.05em} |x|^2} = \left[ \theta \left( \frac{k}{\pi} \right) \right]^N < +\infty,
\]
\noindent where the last inequality comes from the fact that the function \( \theta \) maps \( (0,+\infty) \) into itself. In particular, from the definition of \( \psi \) given in \eqref{definition_psi_lattice} it follows that
\[
\| \psi \|_1 = \sum_{x \in \mathbb{Z}^N} \left| \psi(x) \right| \hspace{0.05em} \mu(x) \equiv \sum_{x \in \mathbb{Z}^N} \psi(x) \hspace{0.05em} \mu(x) = 2 \hspace{0.02em} N \sum_{x \in \mathbb{Z}^N} e^{-k \hspace{0.05em} |x|^2} < +\infty,
\]
\noindent where we have exploited the facts that \( \psi \) is strictly positive, and that \( \mu \equiv 2N \) in \( \mathbb{Z}^N \). This concludes the proof. \hfill \( \square \)

\begin{lemma}\label{lemma_double_sum_lattice}
\noindent For any \( k > 0 \), the function \( \psi \) defined in \eqref{definition_psi_lattice} satisfies
\[
\sum_{x,y \in \mathbb{Z}^N} \omega_0(x,y) \hspace{0.1em} | \nabla_{xy} \psi | < +\infty.
\]
\end{lemma}

\noindent \emph{Proof.} After fixing an arbitrary \( k > 0 \), we first employ the triangular inequality, together with the fact that the function \( \psi \) is strictly positive, in order to infer the following estimate:
\[
\sum_{x,y \in \mathbb{Z}^N} \omega_0(x,y) \hspace{0.1em} | \nabla_{xy} \psi | \equiv \sum_{x,y \in \mathbb{Z}^N} \omega_0(x,y) \hspace{0.1em} | \psi(y) - \psi(x) |
\leq \sum_{x,y \in \mathbb{Z}^N} \omega_0(x,y) \hspace{0.05em} \left[ \psi(y) + \psi(x) \right].
\]
\noindent We now claim the validity of the following identity:
\begin{equation}\label{claim_identity_psi_lattice}
\sum_{x,y \in \mathbb{Z}^N} \omega_0(x,y) \hspace{0.05em} \left[ \psi(y) + \psi(x) \right] = 4 \hspace{0.02em} N \left[ \theta \left( \frac{k}{\pi} \right) \right]^N.
\end{equation}
\noindent Notice that combining this identity with the previous estimate yields:
\[
\sum_{x,y \in \mathbb{Z}^N} \omega_0(x,y) \hspace{0.1em} | \nabla_{xy} \psi | \leq 4 \hspace{0.02em} N \left[ \theta \left( \frac{k}{\pi} \right) \right]^N < +\infty,
\]
\noindent where the last inequality holds since \( N \in \mathbb{N} \) is clearly fixed, and thanks to the fact that the function \( \theta \) maps \( (0,+\infty) \) into itself. Therefore, proving that \eqref{claim_identity_psi_lattice} holds would yield the thesis.

\medskip
\noindent We are then left with showing the validity of identity \eqref{claim_identity_psi_lattice}. First, we can write:
\[
\sum_{x,y \in \mathbb{Z}^N} \omega_0(x,y) \hspace{0.1em} \psi(y) = \sum_{x,y \in \mathbb{Z}^N} \omega_0(y,x) \hspace{0.1em} \psi(x) = \sum_{x,y \in \mathbb{Z}^N} \omega_0(x,y) \hspace{0.1em} \psi(x),
\]
\noindent where we have exchanged the role of the summation variables and exploited the symmetry of the edge weight \( \omega_0 \). This allows us to infer the following identity:
\begin{equation}\label{identity_lemma_double_sum_lattice_1}
\begin{aligned}
\sum_{x,y \in \mathbb{Z}^N} \omega_0(x,y) \hspace{0.05em} \left[ \psi(y) + \psi(x) \right] & = \sum_{x,y \in \mathbb{Z}^N} \omega_0(x,y) \hspace{0.05em} \psi(y) + \sum_{x,y \in \mathbb{Z}^N} \omega_0(x,y) \hspace{0.05em} \psi(x) \\
& = 2 \sum_{x,y \in \mathbb{Z}^N} \omega_0(x,y) \hspace{0.05em} \psi(x).
\end{aligned}
\end{equation}
\noindent We can now rewrite the term at the right-hand side of \eqref{identity_lemma_double_sum_lattice_1} as follows:
\[
\sum_{x,y \in \mathbb{Z}^N} \omega_0(x,y) \hspace{0.05em} \psi(x) \equiv \sum_{x \in \mathbb{Z}^N} \sum_{y \in \mathbb{Z}^N} \omega_0(x,y) \hspace{0.05em} \psi(x) = \sum_{x \in \mathbb{Z}^N} \psi(x) \left[ \sum_{y \in \mathbb{Z}^N} \omega_0(x,y) \right] = 2N \sum_{x \in \mathbb{Z}^N} \psi(x),
\]
\noindent where, in the last equality, we have exploited the identity
\[
\sum_{y \in \mathbb{Z}^N} \omega_0(x,y) = \left| \{ y \in \mathbb{Z}^N : y \sim x \} \right| = 2N, \hspace{2.0em} \text{for all } x \in \mathbb{Z}^N.
\]
\noindent Therefore, thanks to the definition \eqref{definition_psi_lattice} of the function \( \psi \), we obtain:
\[
\sum_{x,y \in \mathbb{Z}^N} \omega_0(x,y) \hspace{0.05em} \psi(x) = 2 \hspace{0.02em} N \sum_{x \in \mathbb{Z}^N} \psi(x)
= 2 \hspace{0.02em} N \sum_{x \in \mathbb{Z}^N} e^{- k \hspace{0.05em} |x|^2} = 2 \hspace{0.02em} N \left[ \theta \left( \frac{k}{\pi} \right) \right]^N,
\]
\noindent where the last equality is due to the result established in \cref{lemma_identity_theta_function}. Combining the obtained identity with \eqref{identity_lemma_double_sum_lattice_1} yields:
\[
\sum_{x,y \in \mathbb{Z}^N} \omega_0(x,y) \hspace{0.05em} \left[ \psi(y) + \psi(x) \right] = 2 \sum_{x,y \in \mathbb{Z}^N} \omega_0(x,y) \hspace{0.05em} \psi(x)
= 4 \hspace{0.02em} N \left[ \theta \left( \frac{k}{\pi} \right) \right]^N.
\]
\noindent In conclusion, identity \eqref{claim_identity_psi_lattice} has been shown, and the thesis follows, as we have already discussed. \hfill \( \square \)

\begin{lemma}\label{lemma_Laplacian_lambda_lattice}
\noindent For any \( k > 0 \), the function \( \psi \) defined in \eqref{definition_psi_lattice} satisfies the following inequality:
\begin{equation}\label{inequality_psi_lattice}
\Delta \psi(x) + \lambda \psi(x) \geq 0 \hspace{2.5em} \text{for all } \hspace{0.1em} x \text{ in } \mathbb{Z}^N,
\end{equation}
\noindent provided that
\[
\lambda \geq 1 - e^{-k}.
\]
\noindent In particular, \eqref{inequality_psi_lattice} holds if \( \lambda \geq 2 \hspace{0.015em} k \hspace{0.015em} N \).
\end{lemma}

\noindent \emph{Proof.} After fixing an arbitrary \( k > 0 \), together with a generic vertex \( x \in \mathbb{Z}^N \), we first aim to rewrite the Laplacian of the function \( \psi \). First, by the definition of both the edge weight \( \omega_0 \) and the node measure \( \mu \), we infer:
\begin{equation}\label{identity_Laplacian_psi_lattice_1}
\begin{aligned}
\Delta \psi(x) & = \frac{1}{\mu(x)} \sum_{y \in \mathbb{Z}^N} \omega_0(x,y) \left[ \psi(y) - \psi(x) \right] \\
& = \frac{1}{2 \hspace{0.02em} N} \sum_{\substack{y \in \mathbb{Z}^N \\ y \sim x}} \left[ \psi(y) - \psi(x) \right] \\
& = \frac{1}{2 \hspace{0.02em} N} \sum_{\substack{y \in \mathbb{Z}^N \\ y \sim x}} \psi(y) - \frac{1}{2 \hspace{0.02em} N} \sum_{\substack{y \in \mathbb{Z}^N \\ y \sim x}} \psi(x) \\
& = \frac{1}{2 \hspace{0.02em} N} \sum_{\substack{y \in \mathbb{Z}^N \\ y \sim x}} \psi(y) - \frac{1}{2 \hspace{0.02em} N} \hspace{0.02em} \psi(x) \sum_{\substack{y \in \mathbb{Z}^N \\ y \sim x}} 1 \\
& = \frac{1}{2 \hspace{0.02em} N} \sum_{\substack{y \in \mathbb{Z}^N \\ y \sim x}} \psi(y) - \psi(x),
\end{aligned}
\end{equation}
\noindent where in the last equality we exploited the identity
\[
\sum_{\substack{y \in \mathbb{Z}^N \\ y \sim x}} 1 = \left| \{ y \in \mathbb{Z}^N : y \sim x \} \right| = 2N.
\]
\noindent Concerning the first term at the right-hand side of \eqref{identity_Laplacian_psi_lattice_1}, we notice that the following identity holds:
\begin{equation}\label{identity_Laplacian_psi_lattice_2}
\sum_{\substack{y \in \mathbb{Z}^N \\ y \sim x}} \psi(y) = \sum_{i=1}^{N} \psi(x - e_i) + \sum_{i=1}^{N} \psi(x + e_i) = \sum_{i=1}^{N} \left[ \psi(x - e_i) + \psi(x + e_i) \right],
\end{equation}
where \( \{ e_i \}_{i=1}^{N} \subset \mathbb{Z}^N \) represents the family of the \emph{canonic vectors}, satisfying, for each \( i,j \in \left\{ 1, 2, \dots, N \right\} \), the definition \( \left( e_i \right)_j := \delta_{ij} \), with \( \delta \) denoting the Kronecker delta. In particular, identity \eqref{identity_Laplacian_psi_lattice_2} can be justified by the fact that the neighbor set of the vertex \( x \) can be written as a disjoint union, namely
\[
\{ y \in \mathbb{Z}^N : y \sim x \} = \bigcup_{i=1}^N \{ x - e_i, x + e_i \} = \left[ \bigcup_{i=1}^N \{ x - e_i \} \right] \cup \left[ \bigcup_{i=1}^N \{ x + e_i \} \right].
\]

\noindent Now, for each \( i = 1, \dots, N \) it holds:
\[
\begin{aligned}
\left| x \pm e_i \right|^2 & = \sum_{j=1}^{N} \left[x_j \pm (e_i)_j \right]^2 \\
& = \sum_{j=1}^{N} \left[ x_j \pm \delta_{ij} \right]^2 \\
& = \sum_{j=1}^{N} \left[ x_j^2 \pm 2 \hspace{0.02em} x_j \hspace{0.02em} \delta_{ij} + \delta_{ij}^2 \right] \\
& = \sum_{j=1}^{N} x_j^2 \pm 2 \sum_{j=1}^{N} x_j \hspace{0.02em} \delta_{ij} + \sum_{j=1}^{N} \delta_{ij}^2 \\
& = |x|^2 \pm 2x_i + 1,
\end{aligned}
\]
\noindent where we have simply exploited the fact that \( (e_i)_j = \delta_{ij} \), together with the definition of both the Kronecker delta and the norm \( |\cdot| \). Therefore, using the definition \eqref{definition_psi_lattice} of the function \( \psi \), we obtain, for all \( i = 1, \dots, N \):
\[
\psi(x \pm e_i) = e^{- k \hspace{0.02em} \left| x \pm e_i \right|^2 } = e^{- k \hspace{0.02em} \left( |x|^2 \pm 2x_i + 1 \right) }
= e^{- k \hspace{0.02em} |x|^2} \hspace{0.05em} e^{\mp 2k x_i} \hspace{0.05em} e^{-k}
= \psi(x) \hspace{0.05em} e^{\mp 2k x_i} \hspace{0.05em} e^{-k}.
\]
\noindent In particular, for any \( i = 1, \dots, N \) we get:
\[
\begin{aligned}
\psi(x - e_i) + \psi(x + e_i) & = \psi(x) \hspace{0.05em} e^{2k x_i} \hspace{0.05em} e^{-k} + \psi(x) \hspace{0.05em} e^{- 2k x_i} \hspace{0.05em} e^{-k} \\
& = \psi(x) \hspace{0.05em} e^{-k} \left[ e^{2k x_i} + e^{- 2k x_i} \right] \\
& = 2 \hspace{0.02em} \psi(x) \hspace{0.05em} e^{-k} \hspace{0.05em} \cosh(2k x_i),
\end{aligned}
\]
\noindent where in the last equality we have exploited the definition of the hyperbolic cosine function. Thus \eqref{identity_Laplacian_psi_lattice_2} yields the following:
\[
\sum_{\substack{y \in \mathbb{Z}^N \\ y \sim x}} \psi(y) = \sum_{i=1}^{N} \left[ \psi(x - e_i) + \psi(x + e_i) \right]
= 2 \hspace{0.02em} \psi(x) \hspace{0.05em} e^{-k} \sum_{i=1}^{N} \hspace{0.05em} \cosh(2k x_i).
\]
\noindent By plugging the obtained identity into \eqref{identity_Laplacian_psi_lattice_1}, we deduce:
\begin{equation}\label{identity_Laplacian_psi_lattice_4}
\begin{aligned}
\Delta \psi(x) & = \frac{1}{2 \hspace{0.02em} N} \sum_{\substack{y \in \mathbb{Z}^N \\ y \sim x}} \psi(y) - \psi(x) \\
& = \frac{1}{N} \hspace{0.02em} \psi(x) \hspace{0.05em} e^{-k} \sum_{i=1}^{N} \hspace{0.05em} \cosh(2k x_i) - \psi(x) \\
& = \psi(x) \left[ \frac{e^{-k}}{N} \sum_{i=1}^{N} \hspace{0.05em} \cosh(2k x_i) - 1 \right].
\end{aligned}
\end{equation}
\noindent Now, since \( \cosh(z) \geq 1 \) for all \( z \in \mathbb{R} \), it easily follows that
\[
\sum_{i=1}^{N} \hspace{0.05em} \cosh(2k x_i) \geq N,
\]
\noindent and by making use of identity \eqref{identity_Laplacian_psi_lattice_4} we get the following inequality:
\[
\Delta \psi(x) \geq \left[ e^{-k} - 1 \right] \psi(x),
\]
\noindent where we have also exploited the fact that \( \psi \) is strictly positive over \( \mathbb{Z}^N \).

\medskip
\noindent Thanks to the arbitrariness of the vertex \( x \), the previous inequality holds pointwise in \( \mathbb{Z}^N \). Thus, for any constant \( \lambda > 0 \), the following holds:
\[
\Delta \psi(x) + \lambda \psi(x) \geq \left[ e^{-k} - 1 + \lambda \right] \psi(x), \hspace{2.0em}  \text{for all } x \in \mathbb{Z}^N.
\]
\noindent Since, again, the function \( \psi \) is strictly positive, it is now immediate to conclude that the desired inequality, namely \eqref{inequality_psi_lattice}, holds if we choose \( \lambda \geq 1 - e^{-k} \). The arbitrariness of \( k > 0 \) then concludes the proof of the first part of the statement.

\medskip
\noindent Finally, since \( N \in \mathbb{N} \), then it trivially follows that \( 2N > 1 \). This fact, together with the well-known estimate
\[
k > 1 - e^{-k} \hspace{2.0em} \text{for all } k > 0,
\]
\noindent allows us to conclude that
\[
2kN > k > 1 - e^{-k}, \hspace{2.0em} \text{for all } k > 0.
\]
\noindent Therefore, by choosing \( \lambda \) such that \( \lambda \geq 2kN > 0 \), in particular it holds \( \lambda \geq 1 - e^{-k} \), and the desired inequality \eqref{inequality_psi_lattice} follows, as discussed above. This yields the thesis. \hfill \( \square \)

\begin{corollary}\label{corollary_varphi_property_B_G_lattice}
\noindent For any \( k > 0 \), the function \( \varphi : \mathbb{Z}^N \to (0,+\infty) \) defined as
\begin{equation}\label{definition_varphi_lattice}
\varphi(x) := C \hspace{0.05em} e^{-k \hspace{0.05em} |x|^2} \hspace{1.0em} \text{for all } \hspace{0.1em} x \in \mathbb{Z}^N, \hspace{1.0em} \text{where} \hspace{1.0em} C := \left\{ 2 \hspace{0.03em} N \left[ \theta \hspace{-0.25em} \left( \frac{k}{\pi} \right) \right]^N \right\}^{-1},
\end{equation}
\noindent satisfies the property \emph{(\(\mathcal{B}_G\))}, provided that
\[
\lambda \geq 1 - e^{-k}.
\]
\noindent In particular, the same result holds if \( \lambda \geq 2 \hspace{0.03em} k \hspace{0.03em} N \).
\end{corollary}

\noindent \emph{Proof.} After fixing an arbitrary \( k > 0 \), we first notice that \( C \in (0,+\infty) \). This is due to both the facts that \( N \in \mathbb{N} \) and that the function \( \theta \) maps \( (0,+\infty) \) into itself. Therefore, \( \varphi \) actually maps \( \mathbb{Z}^N \) into \( (0,+\infty) \), which corresponds to condition a) of property (\(\mathcal{B}_G\)), according to \cref{def_TF_G}.

\medskip
\noindent Moreover, the following relation holds between \( \varphi \) and the function \( \psi \) defined in \eqref{definition_psi_lattice}:
\begin{equation}\label{relation_varphi_psi_lattice}
\varphi(x) = C \hspace{0.1em} \psi(x), \hspace{1.0em} \text{for all } x \in \mathbb{Z}^N.
\end{equation}
\noindent Now, in the proof of \cref{lemma_ell_1_lattice} the function \( \psi \) has already been proved to belong to \( \ell^1(\mathbb{Z}^N,\mu) \). Then, thanks to \eqref{relation_varphi_psi_lattice} we infer that \( \varphi \in \ell^1(\mathbb{Z}^N,\mu) \). Moreover, it holds:
\[
\| \varphi \|_1 = \sum_{x \in \mathbb{Z}^N} \varphi(x) \hspace{0.02em} \mu(x) = 2 \hspace{0.01em} N \hspace{0.01em} C \sum_{x \in \mathbb{Z}^N} e^{-k \hspace{0.05em} |x|^2}
= 2 \hspace{0.01em} N \hspace{0.01em} C \left[ \theta \hspace{-0.25em} \left( \frac{k}{\pi} \right) \right]^N
= 1.
\]
\noindent Here, the first equality is due to the positivity of \( \varphi \), while the second one exploits both \eqref{definition_varphi_lattice} and the fact that \( \mu \equiv 2N \) on \( \mathbb{Z}^N \). The third equality, instead, makes use of the identity established in \cref{lemma_identity_theta_function}. Finally, the last equality is a direct implication of the definition of the constant \( C \) given in \eqref{definition_varphi_lattice}. In conclusion, \( \varphi \) also satisfies condition b) of property (\(\mathcal{B}_G\)).

\medskip
\noindent Now, \cref{lemma_double_sum_lattice} ensures that the function \( \psi \) defined in \eqref{definition_psi_lattice} verifies the following inequality:
\[
\sum_{x,y \in \mathbb{Z}^N} \omega_0(x,y) \hspace{0.1em} | \nabla_{xy} \psi | < +\infty.
\]
\noindent Thanks to the relation \eqref{relation_varphi_psi_lattice}, together with the definition of the difference operator \( \nabla_{xy} \) and the fact that \( C \in (0,+\infty) \), we deduce:
\[
\sum_{x,y \in \mathbb{Z}^N} \omega_0(x,y) \hspace{0.1em} | \nabla_{xy} \varphi | = C \sum_{x,y \in \mathbb{Z}^N} \omega_0(x,y) \hspace{0.1em} | \nabla_{xy} \psi | < +\infty.
\]
\noindent In particular, after fixing two arbitrary values \( R > 0 \) and \( \delta \in (0,1) \), the sum at the left-hand side of the previous estimate is independent of both \( R \) and \( \delta \), and the following holds:
\[
\sum_{\substack{x,y \in \mathbb{Z}^N \\ \hspace{1.0em} x \in \overline{B}_R \setminus B_{(1-\delta)R - 2s}}} \omega(x,y) \hspace{0.2em} |\nabla_{xy} \varphi| \hspace{0.1em}
\leq \hspace{0.1em} \sum_{x,y \in \mathbb{Z}^N} \omega(x,y) \hspace{0.1em} | \nabla_{xy} \varphi | < +\infty,
\]
\noindent which implies that also condition c) of property (\(\mathcal{B}_G\)) is satisfied by \( \varphi \).

\medskip
\noindent Finally, if we choose \( \lambda \) such that \( \lambda \geq 1 - e^{-k} \), then \cref{lemma_Laplacian_lambda_lattice} ensures that
\[
\Delta \psi(x) + \lambda \psi(x) \geq 0 \hspace{1.8em} \text{for all } \hspace{0.1em} x \text{ in } \mathbb{Z}^N.
\]
\noindent By exploiting both the linearity of the Laplacian operator and the relation \eqref{relation_varphi_psi_lattice}, we finally obtain, for all \( x \in \mathbb{Z}^N \):
\[
\Delta \varphi(x) + \lambda \varphi(x) = C \left[ \Delta \psi(x) + \lambda \psi(x) \right] \geq 0,
\]
\noindent where the inequality is due to the fact that \( C \in (0,+\infty) \). In conclusion, if \( \lambda \geq 1 - e^{-k} \), then the function \( \varphi \) satisfies condition d) of property (\(\mathcal{B}_G\)), according to \cref{def_TF_G}. The arbitrariness of \( k > 0 \) then concludes the proof of the first part of the statement.

\medskip
\noindent We can now argue as in the last part of the proof of \cref{lemma_Laplacian_lambda_lattice}, in order to obtain the following estimate:
\[
2kN > k > 1 - e^{-k}, \hspace{2.0em} \text{for all } k > 0.
\]
\noindent Therefore, by choosing \( \lambda \) such that \( \lambda \geq 2kN > 0 \), in particular it holds \( \lambda \geq 1 - e^{-k} \), and the function \( \varphi \) satisfies the property (\(\mathcal{B}_G\)), as discussed above. This yields the thesis. \hfill \( \square \)

\medskip
\noindent \emph{Proof of \emph{\cref{theorem_blow_up_lattice_1}}}. We start by recalling that, by \cref{lemma_assumption_3_1_lattice}, the integer lattice \( (\mathbb{Z}^N,\omega_0,\mu) \) satisfies \cref{assumption_graph_pseudo_metric}, for any fixed dimension \( N \in \mathbb{N} \). Moreover, by \cref{corollary_varphi_property_B_G_lattice}, the function \( \varphi \) defined in \eqref{definition_varphi_lattice} satisfies property (\(\mathcal{B}_G\)), according to \cref{def_TF_G}, provided that \( \lambda \geq 1 - e^{-k} \), and in particular if \( \lambda \geq 2 \hspace{0.03em} k \hspace{0.03em} N \).

\medskip
\noindent Now, thanks to the identity
\[
\deg(x) = \mu(x) = 2N \hspace{2.5em} \text{for all } x \in \mathbb{Z}^N,
\]
\noindent we infer that the weighted degree function, introduced in \cref{definition_weighted_degree}, is actually a constant; indeed, it holds:
\[
\operatorname{Deg}(x) = \frac{\deg(x)}{\mu(x)} \equiv 1 \hspace{2.5em} \text{for all } x \in \mathbb{Z}^N.
\]
\noindent In particular, we have:
\[
\sup_{x \in \mathbb{Z}^N} \operatorname{Deg}(x) = 1 < +\infty,
\]
\noindent meaning that the hypothesis in \eqref{assumptions_graph_hypothesis_verified} is here satisfied, and thus by \cref{remark_proposition_blow_up_chapter_3} we deduce that the integrability condition \eqref{assumption_proposition_blow_up_chapter_3} in the statement of \cref{theorem_blow_up_chapter_3} is automatically true.

\medskip
\noindent We are then allowed to argue as in the proof of \cref{theorem_blow_up_chapter_3}. More specifically, by contradiction we assume that \( u \not \equiv 0 \) is a global solution of problem \eqref{problem_HR_G}, namely \( u(\cdot, t) \in \ell^\infty(\mathbb{Z}^N) \) for all \( t > 0 \). In addition, we suppose that \eqref{e20f} holds.

\medskip
\noindent Therefore, we can repeat the exact same passages as in the proof of \cref{theorem_blow_up_chapter_3}, obtaining the following first order differential problem:
\begin{equation}\label{CP_G}
\begin{cases}
\Phi'(t) + \lambda \Phi(t) \geq f(\Phi(t)) \hspace{1.5em} \text{for all } t > 0 \\
\Phi(0) = \sum \limits_{x \in \mathbb{Z}^N} \varphi(x) \hspace{0.05em} u_0(x) \hspace{0.05em} \mu(x),
\end{cases}
\end{equation}
\noindent having as unknown the function \( \Phi : [0,+\infty) \to [0,+\infty) \) defined as
\[
\Phi(t) := \sum_{x \in \mathbb{Z}^N} \varphi(x) \hspace{0.1em} u(x,t) \hspace{0.1em} \mu(x) \hspace{1.6em} \text{for all } t \geq 0.
\]
\noindent As already discussed in the proof of \cref{theorem_blow_up_chapter_3}, if \( \Phi(0) > s_0(\lambda) \), then \( \Phi \) exhibits a vertical asymptote at some finite time. But this would imply a contradiction, since our assumptions on \( \varphi, \mu, u, u_0 \) yield:
\[
0 \leq \Phi(t) < +\infty \hspace{1.6em} \text{for all } t \geq 0.
\]
\noindent Therefore, if it holds \( \Phi(0) > s_0(\lambda) \), we conclude that \( u \) cannot be global, meaning that it must blow up in finite time, which is the thesis. Now, it easily seen that
\begin{equation}\label{Phi_0_lattice}
\Phi(0) = \sum_{x \in \mathbb{Z}^N} \varphi(x) \, u_0(x) \, \mu(x)
= \frac{1}{2N} \left[ \theta \hspace{-0.25em} \left( \frac{k}{\pi} \right) \right]^{-N}
\sum_{x \in \mathbb{Z}^N} e^{-k \hspace{0.05em} |x|^2} \, u_0(x) \, \mu(x),
\end{equation}
where the last equality exploits both \eqref{definition_varphi_lattice} and the fact that \( \mu \equiv 2N \). This means that the condition \( \Phi(0) > s_0(\lambda) \) is equivalent to \eqref{e20f}, which completes the proof of the first part of the statement.

\medskip
\noindent Now, we assume that $f(u)=u^p \;\, (p>1)$. We then need to find an equivalent way of writing the inequality \( \Phi(0) > s_0(\lambda)=\lambda^{\frac{1}{p-1}} \). Notice that the positive constant \( \lambda \) appearing in \eqref{CP_G} must verify
\[
\Delta \varphi + \lambda \varphi \geq 0 \hspace{1.6em} \text{pointwise in } \mathbb{Z}^N,
\]
\noindent and from \cref{corollary_varphi_property_B_G_lattice} we know that this inequality holds if we choose \( \lambda \geq 2kN > 0 \), as already recalled. Then we directly set \( \lambda = 2kN \), so that
\[
\lambda^{\frac{1}{p-1}} = (2kN)^{\frac{1}{p-1}}.
\]
Therefore, by exploiting \eqref{Phi_0_lattice}, we obtain the validity of the following equivalences:
\[
\begin{aligned}
\Phi(0) > \lambda^{\frac{1}{p-1}} &\iff \frac{1}{2N} \left[ \theta \left( \frac{k}{\pi} \right) \right]^{-N}
\sum_{x \in \mathbb{Z}^N} e^{-k \hspace{0.05em} |x|^2} \, u_0(x) \, \mu(x) > (2kN)^{\frac{1}{p-1}} \\
&\iff \sum_{x \in \mathbb{Z}^N} e^{-k \hspace{0.05em} |x|^2} \, u_0(x) \, \mu(x) >
(2N)^{\frac{p}{p-1}} \left[ \theta \left( \frac{k}{\pi} \right) \right]^N k^{\frac{1}{p-1}}.
\end{aligned}
\]
\noindent We notice that the last inequality corresponds exactly to \eqref{condition_blow_up_G_lemma_1}. In conclusion, under \eqref{condition_blow_up_G_lemma_1} it holds \( \Phi(0) > \lambda^{\frac{1}{p-1}} \). As already remarked above, this yields the thesis. \hfill \( \square \)



\addcontentsline{toc}{section}{Bibliography}

\begin{thebibliography}{99}


\bibitem{BB}
C. Bandle, H. Brunner, \emph{Blowup in diffusion equations: a survey}, J. Comput. Appl. Math. \textbf{97} (1998), 3–22.


\bibitem{BLev}
C. Bandle, H. A. Levine, \emph{Fujita type phenomena for reaction‑diffusion equations with convection like terms}, Diff. Integral Eq. \textbf{7} (1994), 1169–1193.


\bibitem{BPT}
C. Bandle, M.A. Pozio,  A. Tesei, \emph{The Fujita exponent for the Cauchy problem in the hyperbolic space}, J. Differ. Equ. \textbf{251} (2011), 2143–2163.




\bibitem{BMP}
S. Biagi, G. Meglioli,  F. Punzo, \emph{A Liouville theorem for elliptic equations with a potential on infinite graphs}, Calc. Var. Part. Diff. Eq. \textbf{63}, 165 (2024).


\bibitem{BP}
S. Biagi, F. Punzo, \emph{Phragmén–Lindelöf type theorems for elliptic equations on infinite graphs}, Potential Anal. {\bf 64} 19 (2026).






\bibitem{vCr} D.E. von Criegern, \emph{Nonexistence results for a general class of parabolic problems with a potential on weighted graphs}. Nonlinear Differ. Equ. Appl. {\bf 33}, 44 (2026).


\bibitem{DL}
K. Deng, H.A. Levine, \emph{The role of critical exponents in blow-up theorems: the sequel}, J. Math. Anal. Appl. \textbf{243} (2000), 85–126.


\bibitem{Elk}
N. D. Elkies, \emph{Theta functions and weighted theta functions of Euclidean lattices, with some applications}, (2009).




\bibitem{FR}
F. Fischer, C. Rose, \emph{Optimal Poincaré–Hardy-type inequalities on manifolds and graphs}, Indagationes Mathematicae, Elsevier (2025).






\bibitem{Fujita}
H. Fujita, \emph{On the blowing-up of solutions of the Cauchy problem for \( u_t = \Delta u + u^{1+\alpha} \)}, J. Fac. Sci. Univ. Tokyo Sect. IA Math. \textbf{16} (1966), 105–113.


\bibitem{Grig1}
A. Grigor’yan, \emph{Introduction to analysis on graphs}, AMS University Lecture Series \textbf{71}, 2018.




\bibitem{GMP1}
G. Grillo, G. Meglioli, F. Punzo, \emph{Blow-up versus global existence of solutions for reaction–diffusion equations on classes of Riemannian manifolds}, Ann. Mat. Pura Appl. \textbf{202} (2023), 1255–1270.








\bibitem{GMP2}
G. Grillo, G. Meglioli, F. Punzo, \emph{Blow-up and global existence for semilinear parabolic equations on infinite graphs},
Calc. Var. Part. Diff. Eq. {\bf 65} 114 (2026).



\bibitem{GSXX} Q. Gu, Y. Sun, J. Xiao, F. Xu, {\it Global positive solution to a semilinear parabolic equation with
potential on Riemannian manifold}, Calc. Var. Partial Diff. Eq. {\bf 59} 170 (2020).

\bibitem{HK} B. Hua, M. Keller, {\it Harmonic functions of general graph Laplacians}. Calc. Var. Part. Diff. Eq. {\bf 51}
(2014), 343–362.

\bibitem{H}
K. Hayakawa, \emph{On nonexistence of global solutions of some semilinear parabolic differential equations}, Proc. Japan Acad. \textbf{49} (1973), 503–505.


\bibitem{Huang}
X. Huang, \emph{On uniqueness class for a heat equation on graphs}, J. Math. Anal. Appl. \textbf{393} (2012), 377–388.




\bibitem{KLWb}
M. Keller, D. Lenz,  R.K. Wojciechowski, \emph{Graphs and discrete Dirichlet spaces}, Springer, 2021.


\bibitem{KLW}
M. Keller, D. Lenz,  R. K. Wojciechowski, \emph{Volume growth, spectrum and stochastic completeness of infinite graphs}, Math. Z. \textbf{274} (2013), 905–932.


\bibitem{KST}
K. Kobayashi, T. Sirao,  H. Tanaka, \emph{On the growing up problem for semilinear heat equations}, J. Math. Soc. Japan \textbf{29} (1977), 407–424.


\bibitem{LSZ} D. Lenz, M. Schmidt, I. Zimmermann, {\it Blow-up of nonnegative solutions of an abstract semilinear
heat equation with convex source} Calc. Var. Part. Diff. Eq.  {\bf 62} 140 (2023).

\bibitem{Lev}
H.A. Levine, \emph{The role of critical exponents in blowup theorems}, SIAM Rev. \textbf{32} (1990), 262–288.


\bibitem{LW1}
Y. Lin, Y. Wu, \emph{Blow-up problems for nonlinear parabolic equations on locally finite graphs}, Acta Math. Sci. \textbf{38B} (3) (2018), 843–856.


\bibitem{LW2}
Y. Lin, Y. Wu, \emph{The existence and nonexistence of global solutions for a semilinear heat equation on graphs}, Calc. Var. Part. Diff. Eq. \textbf{56}, 102 (2017).




\bibitem{MMP}
P. Mastrolia, D.D. Monticelli,  F. Punzo, \emph{Nonexistence of solutions to parabolic differential inequalities with a potential on Riemannian manifolds}, Math. Ann. \textbf{367} (2017), 929–963.


\bibitem{MePu}
G. Meglioli, F. Punzo, \emph{Uniqueness in weighted \(\ell^p\) spaces for the Schrödinger equation on infinite graphs}, Proc. Amer. Math. Soc. \textbf{153} (2025), 1519-1537.


\bibitem{MitPo}
E. Mitidieri, S.I. Pohozaev, \emph{Towards a unified approach to nonexistence of solutions for a class of differential inequalities}, Milan J. Math. \textbf{72} (2004), 129–162.


\bibitem{MPS1} D. D. Monticelli, F. Punzo, J. Somaglia, {\it Nonexistence results for semilinear elliptic equations on weighted graphs}, Math. Ann. {\bf 393}  (2025), 3395–3418.

\bibitem{MPS2}
D. D. Monticelli, F. Punzo,  J. Somaglia, \emph{Nonexistence of solutions to parabolic problems with a potential on weighted graphs}, J. Differential Equations \textbf{453} (2026), 113782.


\bibitem{Mu} D. Mugnolo, “Semigroup Methods for Evolution Equations on Networks”, Springer (2016).



\bibitem{Pu1}
F. Punzo, \emph{Blow-up of solutions to semilinear parabolic equations on Riemannian manifolds with negative sectional curvature}, J. Math. Anal. Appl. \textbf{387} (2012), 815–827.


\bibitem{Pu2}
F. Punzo, \emph{Global solutions of semilinear parabolic equations with drift term on Riemannian manifolds}, Discrete Contin. Dyn. Syst. \textbf{42} (2022), 3733–3746.


\bibitem{PSa}
F. Punzo, A. Sacco, \emph{On a semilinear parabolic equation with time-dependent source term on infinite graphs},  J. Evol. Equ. {\bf 26} 13 (2026).

\bibitem{PZ2} F. Punzo, F. Zucchero, {\it On a semilinear heat equation on infinite graphs II: blow-up for arbitrary initial data and global existence}, preprint (2026).





\bibitem{W}
L. F. Wang, \emph{Heat kernel and monotonicity inequalities on the graph}, J. Geom. Anal. \textbf{33}, 38 (2023).




\bibitem{Weis}
F.B. Weissler, \emph{Existence and nonexistence of global solutions for a semilinear heat equation}, Israel J. Math. \textbf{38} (1981), 29–40.


\bibitem{Woj}
R. K. Wojciechowski, \emph{Heat kernel and essential spectrum of infinite graphs}, Indiana Univ. Math. J. \textbf{58}, no. 3 (2009), 1419–1441.


\bibitem{Wu}
Y. Wu, \emph{On nonexistence of global solutions for a semilinear heat equation on graphs}, Nonlinear Anal. \textbf{171} (2018), 73–84.


\bibitem{Z}
Q.S. Zhang, \emph{Blow-up results for nonlinear parabolic equations on manifolds}, Duke Math. J. \textbf{97} (3) (1999), 515–539.


\end{thebibliography}
\end{document}